\documentclass[12pt]{article}
\usepackage{enumerate}
\usepackage{amsmath,amssymb,amsthm, mathrsfs, mathtools}
\usepackage{latexsym}
\usepackage{graphics,graphicx}
\usepackage{hyperref}
\usepackage[bf, small]{titlesec}
\usepackage{authblk} 
\usepackage{soul}

\usepackage{lineno}

\theoremstyle{plain}
\newtheorem{Thm}{Theorem}[section]
\newtheorem{XThm}{Theorem}
\newtheorem{Lem}[Thm]{Lemma}
\newtheorem{Prop}[Thm]{Proposition}
\newtheorem{Cor}[Thm]{Corollary}

\theoremstyle{definition}

\newtheorem{Rmk}[Thm]{Remark}

\usepackage{standalone}
\usepackage{tikz}
\usetikzlibrary{arrows,positioning,matrix,fit,backgrounds,shapes,shapes.geometric, intersections}

\usetikzlibrary{shapes.callouts,decorations.pathmorphing} 
\usetikzlibrary{shadows} 
\usepackage{calc} 
\usetikzlibrary{decorations.markings} 
\tikzstyle{vertex}=[circle, draw, inner sep=0pt, minimum size=6pt] 
\newcommand{\vertex}{\node[vertex]}
\usetikzlibrary{arrows,matrix} 
\title{Competitively orientable complete multipartite graphs}

\author[1]{\small Myungho Choi}
\author[1]{\small Minki Kwak}
\author[1]{\small Suh-Ryung Kim}

\affil[1]{\footnotesize Department of Mathematics Education, Seoul National University, Seoul 08826}
\affil[ ]{\footnotesize\textit{nums8080@naver.com, limpkmk@naver.com, srkim@snu.ac.kr}}
\date{}
\begin{document}
\maketitle
\begin{abstract}
    We say that a digraph $D$ is competitive if any pair of vertices has a common out-neighbor in $D$ and that a graph $G$ is competitively orientable if there exists a competitive orientation of $G$. The notion of competitive digraphs arose while studying digraph whose competition graphs are complete.
    We derive some useful properties of competitively orientable graphs and show that a complete graph of order $n$ is competitively orientable if and only if $n \geq 7$.
    Then we completely characterize a competitively orientable complete multipartite graph in terms of the sizes of its partite sets.
  Moreover, we present a way to build a competitive multipartite tournament in each of competitively orientable cases.
    \end{abstract}



\section{Introduction}
In this paper, for graph-theoretical terminology and notations not defined, we follow \cite{bondy}. We consider finite simple graphs.
For a digraph $D$,
the \emph{underlying graph} of $D$
is the graph $G$
such that $V(G)=V(D)$ and $E(G)=\{ uv \mid (u,v) \in A(D) \}$.
An \emph{orientation} of a graph $G$
is a digraph having no directed $2$-cycles, no loops, and no multiple arcs
whose underlying graph  is $G$.

We say that two vertices {\it compete} in a digraph $D$ if they have a common out-neighbor in $D$ and that a digraph $D$ is {\it competitive} if any pair of vertices competes in $D$.
A graph $G$ is said to be {\it competitively orientable} if there exists a competitive  orientation of $G$.
For example, the complete graph $K_7$ is completely orientable as shown in Figure~\ref{fig:7-tournament-complete}.
\begin{figure}
\begin{center}
 \begin{tikzpicture}[auto,thick]
       \tikzstyle{player}=[minimum size=5pt,inner sep=0pt,outer sep=0pt,draw,circle]

    \tikzstyle{player1}=[minimum size=2pt,inner sep=0pt,outer sep=0pt,fill,color=black, circle]
    \tikzstyle{source}=[minimum size=5pt,inner sep=0pt,outer sep=0pt,ball color=black, circle]
    \tikzstyle{arc}=[minimum size=5pt,inner sep=1pt,outer sep=1pt, font=\footnotesize]
    \path (90:1.5cm)     node [player]  (v0)
    {};
    \path (40:1.5cm)     node [player]  (v1)  {};
    \path (-10:1.5cm)     node [player]  (v2){};
    \path (-60:1.5cm)     node [player]  (v3){};
 \path (140:1.5cm)     node [player]  (v6){};
  \path (190:1.5cm)     node [player]  (v5){};
    \path (240:1.5cm)     node [player]  (v4){};

   \draw[black,thick,-stealth] (v0) - +(v1);
 \draw[black,thick,-stealth] (v0) - +(v2);
  \draw[black,thick,-stealth] (v0) - +(v4);

  \draw[black,thick,-stealth] (v1) - +(v2);
  \draw[black,thick,-stealth] (v1) - +(v3);
  \draw[black,thick,-stealth] (v1) - +(v5);

  \draw[black,thick,-stealth] (v2) - +(v3);
  \draw[black,thick,-stealth] (v2) - +(v4);
  \draw[black,thick,-stealth] (v2) - +(v6);

  \draw[black,thick,-stealth] (v3) - +(v4);
  \draw[black,thick,-stealth] (v3) - +(v0);
  \draw[black,thick,-stealth] (v3) - +(v5);

  \draw[black,thick,-stealth] (v4) - +(v1);
  \draw[black,thick,-stealth] (v4) - +(v5);
  \draw[black,thick,-stealth] (v4) - +(v6);

  \draw[black,thick,-stealth] (v5) - +(v2);
  \draw[black,thick,-stealth] (v5) - +(v0);
  \draw[black,thick,-stealth] (v5) - +(v6);

  \draw[black,thick,-stealth] (v6) - +(v1);
  \draw[black,thick,-stealth] (v6) - +(v3);
  \draw[black,thick,-stealth] (v6) - +(v0);

    \end{tikzpicture}
 \end{center}
 \caption{A competitive orientation of $K_7$}
\label{fig:7-tournament-complete}
\end{figure}
By the way, we deduce Theorem~\ref{thm:making-multi-complete}
which guarantees $K_n$ being competitively orientable for $n \geq 7$ from $K_7$ being competitively orientable.
Yet, $K_n$ is not competitively orientable for any integer $2 \le n \le 6$ by Theorem~\ref{thm:char-orient-graph}(3).

The notions of competitive digraph and competitive orientation arose during a research on competition graphs of complete multipartite graphs.
The {\it competition graph} of a digraph $D$ is defined as the graph with the vertex set $V(D)$ and an edge $uv$ if and only if $u$ and $v$ compete in $D$.
Competition graphs arose in
connection with an application in ecology (see \cite{Cohen})
and also have applications in coding,
radio transmission, and modeling of complex economic systems.
Early literature of the study on competition graphs is summarized
in the survey papers by Kim~\cite{kim1993competition} and Lundgren~\cite{lundgren1989food}.
The competition graphs of tournaments and those of bipartite tournaments
have been actively studied
(see \cite{cho2002domination}, \cite{choi20171},  \cite{eoh2019niche},
  \cite{eoh2020m},
 \cite{factor2007domination},
  \cite{fisher2003domination},
  \cite{fisher1998domination}, and \cite{kim2016competition}
for papers related to this topic).

By the definition of competition graph, it is easy to see that a digraph is competitive if and only if its competition graph is a complete graph.
On the other hand, the competition graph of a digraph $D$ being complete may be rephrased as: The adjacency matrix of a digraph $D$ is ``scrambling".
A matrix $A$ is said to be {\it scrambling} if for any pair of indices $i$, $j$, there exists $k$ such that $A_{ik}\neq 0$ and $A_{jk} \neq 0$.
Scrambling matrices were first defined in \cite{hajnal1958weak} to study weak ergodicity of inhomogeneous Markov chains.

Kim and Lee \cite{lee2009competition} studied acyclic digraphs whose competition graphs consist of only complete components.

In this paper, we completely characterize a competitively orientable complete multipartite graph in terms of the sizes of its partite sets.
We first show that there is no competitively orientable complete bipartite graph (Corollary~\ref{cor:no-biparte-complete}).
Then we show that for each integer $k \ge 7$, any $k$-partite complete graph is competitively orientable (Proposition~\ref{prop:k>=7-complete}).
Next we characterize competitively orientable complete $6$-partite graphs as follows.
\begin{XThm} \label{thm:complete-6-partite}
  Let $n_1,\ldots,n_6$ be positive integers such that $n_1 \geq \cdots \geq n_6$.
Then a complete $6$-partitie graph $K_{n_1,n_2,\ldots,n_6}$ is competitively orientable if and only if one of the following holds: (a) $n_1\geq5$ and $n_2=1$; (b) $n_1 \geq 3$, $n_2 \geq 2$, and $n_3=1$; (c) $n_3\geq 2$.
  \end{XThm}
  The remaining cases are also completely taken care of in the following manner.
\begin{XThm} \label{thm:complete-3-partite}
  Let $n_1$, $n_2$, and $n_3$ be positive integers such that $n_1 \geq n_2 \geq n_3$.
  Then a complete tripartite graph $K_{n_1,n_2,n_3}$ is competitively orientable if and only if $n_1\geq5$ and $n_3 \geq 4$.
  \end{XThm}
  \begin{XThm} \label{thm:complete-4-partite}
    Let $n_1,\ldots,n_4$ be positive integers such that $n_1 \geq \cdots \geq n_4$.
    Then a complete $4$-partite graph $K_{n_1,n_2,n_3,n_4}$ is competitively orientable if and only if one of the following holds:
    (a) $n_1 \geq 4$, $n_3 \geq 3$, and $n_4 = 1$;
    (b)  $n_1 \geq 4$, $n_3 =2$, and $n_4 =2$;
    (c) $ n_3 \geq 3$ and $n_4 \geq 2$.
    \end{XThm}
  \begin{XThm} \label{thm:complete-5-partite}
    Let $n_1,\ldots,n_5$ be positive integers such that $n_1 \geq n_2 \geq \cdots \geq n_5$. Then a complete $5$-partite graph $K_{n_1,n_2,\ldots,n_5}$ is competitively orientable if and only if one of the following holds:   (a) $n_1=3$, $n_2=3$, $n_3 \geq 2$, and $n_4=1$;
    (b) $n_1\geq 4$, $n_3 \geq 2$, and $n_4 =1$;
    (c) $n_4 \geq 2$.
    \end{XThm}

    A \emph{tournament} is an orientation of a complete graph.
A \emph{$k$-partite tournament} is  an orientation of a complete $k$-partite graph for some positive integer $k \geq 2$.
If a digraph is a $k$-partite tournament for some integer $k \geq 2$, then it is called a \emph{multipartite tournament}.
Multipartite tournaments have been actively studied by  graph theorists (see \cite{figueroa2012acyclic},  \cite{figueroa2016strong}, \cite{galeana2011k}, \cite{guo2012weakly}, and a survey paper \cite{volkmann2007multipartite}).

We make a useful observation that any complete multipartite graph containing a competitively orientable complete multipartite graph as a subgraph is competitively orientable (Corollary~\ref{lem:making-multi-complete}).
Thanks to this observation, showing a complete multipartite graph $K_{n_1,n_2,\ldots,n_k}$ is competitively orientable becomes much simpler: once we present a deliberately designed concrete competitive multipartite tournament as a base, the proposition guarantees that it will expand to a competitive multipartite tournament with partite sets of sizes $n_1,n_2,\ldots, n_k$.

In Section~\ref{sec:preliminaries}, we derive some properties of competitively orientable graphs which are useful in proving our main results.
In Section~\ref{sec:6-partite}, we study structure of competitive $6$-partite tournament to prove Theorem~\ref{thm:complete-6-partite}.
In Section~\ref{sec:3-partite}, we deal with competitive tripartite tournaments and prove Theorem~\ref{thm:complete-3-partite}.
In Section~\ref{sec:4,5-partite}, we prove Theorems~\ref{thm:complete-4-partite} and \ref{thm:complete-5-partite}.
In Section~\ref{sec:completecompetition}, we summarize the main results in the aspect of vertices of  the competition graph of a competitive multipartite tournament (Theorem~\ref{Thm:summary:K_n-k-partite}).

\section{Preliminaries} \label{sec:preliminaries}
In this section, we derive properties of competitive digraphs and competitively orientable graphs, and develop tools to prove our mains results.
\subsection{Competitively orientable graphs}
Given a digraph $D$, we denote by $N^+_D(x)$ the set of out-neighbors of a vertex $x$ in $D$ and by $N^-_D(x)$ the set of in-neighbors of a vertex $x$ in a digraph $D$. If no confusion is likely, we omit the subscript $D$.

Given a digraph $D$ and a vertex $u$ of $D$,
we add a new vertex $v$ and arcs to $D$ including the arcs in the set $\{(v,w) \mid (u,w) \in A(D)\}$.
We call the resulting digraph
a {\it digraph competitively expanded from $D$ via $u$ by $v$}.

\begin{Prop} \label{prop:adding-vertex}
Given a nontrivial competitive digraph $D$ and a vertex $u$ in $D$, each digraph competitively expanded from $D$ via $u$ by a new vertex $v$ is competitive.
\end{Prop}

\begin{proof}
Let $D'$ be a digraph competitively expanded from $D$ via $u$ by a new vertex $v$.
By the hypothesis, $A(D)\subset A(D')$ and $N^+_D(u) \subseteq N^+_{D'}(v)$.
Take two vertices $x$ and $y$ in $D'$.
If $x \neq v $ and $y \neq v$, then $x$ and $y$ compete in $D$ and so in $D'$. By symmetry, now we suppose $x=v$.
Then $y \neq v$, so $y \in V(D)$.
Since $D$ is a nontrivial competitive digraph, $N^+_D(y) \cap N^+_D(u) \neq \emptyset$.
Then $N^+_{D'}(x) \cap N^+_{D'}(y)=N^+_{D'}(v) \cap N^+_{D'}(y) \supseteq N^+_D(u) \cap N^+_D(y) \neq \emptyset$ and so $x$ and $y$ compete in $D'$.
\end{proof}

\begin{Prop} \label{prop:char-competi-digraph}
Let $D$ be a competitive digraph.
Then the following are true:
\begin{enumerate}[{(1)}]

 \item if $D$ is a nontrivial digraph and has a vertex $v$ of indegree at most $1$, then $D-v$ is competitive;

\item if $D_v$ is the subdigraph of $D$ induced by $N^+(v)$ for a vertex $v$ in $D$, then each vertex has outdegree at least one in $D_v$;

\item each vertex in $D$ has outdegree at least $3$, especially, if a vertex $u$ has outdegree $3$ in $D$,
then its out-neighbors form a directed cycle;

\item there exist at least $\max \{4|V(D)|-|A(D)|,0\}$ vertices of outdegree $3$ in $D$.
\end{enumerate}
\end{Prop}
\begin{proof}
The statement (1) is obviously true.

To show the statement (2),
take a vertex $v$ in $D$.
Let $D_v$ is the subdigraph of $D$ induced by $N^+(v)$.
If there exists an out-neighbor $w$ of $v$ which has out-degree $0$ in $D_v$, then $v$ and $w$ cannot compete in $D$, a contradiction.
The statement (3) is an immediate consequence of the statement (2).

Let $l$ be the number of vertices of outdegree $3$.
Since each vertex in $D$ has outdegree at least $3$ by the statement (3),
\[4(|V(D)|-l)+3l \leq |A(D)|.\]
Therefore $4|V(D)|-|A(D)| \leq l$.
Thus the statement (4) is true.
\end{proof}

The following is an immediate consequence of Proposition~\ref{prop:char-competi-digraph}(2).

\begin{Cor} \label{cor:no-biparte-complete}
There is no competitive bipartite tournament.
\end{Cor}

Proposition~\ref{prop:char-competi-digraph}
may be rephrased as graph version in the following.

\begin{Thm} \label{thm:char-orient-graph}
Let $G$ be a competitively orientable graph.
Then the following are true:
\begin{enumerate}[{(1)}]
\item each vertex in $G$ has at least three neighbors, especially, if a vertex has exactly three neighbors, then its neighbors form a clique;

\item if $G$ is nontrivial and has a vertex $v$ of degree at most $4$,
then $G-v$ is a competitively orientable graph;

\item $|V(G)| \geq 7$ and $|E(G)| \geq 3 |V(G)|$.

\end{enumerate}
\end{Thm}
\begin{proof}
Let $D$ be a competitive orientation of $G$.
Then the statement (1) is immediately true by Proposition~\ref{prop:char-competi-digraph}(3).

To show the statement (2), suppose there exists a vertex $v$ of degree at most $4$.
Then, by Proposition~\ref{prop:char-competi-digraph}(3), $v$ has indegree at most $1$.
Therefore $D-v$ is competitive by Proposition~\ref{prop:char-competi-digraph}(1).
Thus $G-v$ is competitively orientable and so the statement (2) is true.

Since each vertex in $D$ has outdegree at least $3$ by Proposition~\ref{prop:char-competi-digraph}(3),
$|A(D)| \geq 3 |V(D)|$ and so $|E(G)| \geq 3 |V(G)|$.
By the way, since $G$ is simple, $|E(G)| \leq \binom{|V(G)|}{2}$.
Therefore
\[ 3 |V(G)| \leq \frac{|V(G)|(|V(G)|-1)}{2}.\]
Thus $|V(G)| \geq 7$.
\end{proof}

\begin{Rmk} \label{rmk:edges-vertices}
The inequality $|V(G)| \geq 7$ and $|E(G)| \geq 3 |V(G)|$ given in Theorem~\ref{thm:char-orient-graph} is tight. By the way, for each integer $m \geq 7$, there exists a competitively orientable graph of order $m$ with $3m$ edges.
\end{Rmk}
\begin{proof}
Take an integer $m \geq 7$ and a digraph $D$ given in Figure~\ref{fig:7-tournament-complete}, which is a competitive orientation of $K_7$.
We note that each vertex in $D$ has outdegree $3$.
Since $|E(K_7)|=21=3|V(K_7)|$, $K_7$ is the desired one for $m=7$.
Now we assume $m \geq 8$.
We begin with $D$ to construct a desired digraph.
Take a vertex $u$ in $D$.
Then $|N_{D}^+(u)|=3$.
Inductively, we identify $D_0$ with $D_0$ and competitively expand $D_i$ from $D_{i-1}$ via $u$ by a new vertex $v_i$ so that $N^+_{D_i}(v_i)=N^+_{D}(u)$ and $N^-_{D_i}(v_i)=\emptyset$ for each $1\leq i \leq m-7$.
Then $|A(D_i)|=|A(D)|+3i=3(7+i)=3|V(D_i)|$
and $D_i$ is competitive for each $1\leq i \leq m-7$ by Proposition~\ref{prop:adding-vertex}.
Therefore the underlying graph of $D_{m-7}$ is the desired one.
\end{proof}
We make a useful observation as follows.
\begin{Thm} \label{thm:making-multi-complete}
Let $G$ be a competitively orientable graph and $G'$ be a supergraph of $G$ such that for each vertex $v$ in $G'$, there exists a vertex $u$ in $G$ satisfying $N_G(u) \subset N_{G'}(v)$. Then $G'$ is also competitively orientable.
\end{Thm}
\begin{proof}
Suppose that $D$ is a competitive orientation of $G$.
If $V(G')= V(G)$, then each orientation $D'$ of $G'$ obtained by orienting edges in $E(G')\setminus E(G)$ arbitrarily so that $A(D) \subset A(D')$ is competitive.

Suppose $V(G') \neq V(G)$.
Then $V(G')\setminus V(G)=\{v_1,\ldots,v_k\}$ for a positive integer $k$. By the hypothesis, there exists a vertex $u_i$ in $G$ such that $N_G(u_i) \subset N_{G'}(v_i)$ for each $1\leq i \leq k $.
Let $G_0=G$, $D_0=D$, and $G_i=G'[V(G) \cup \{v_1,\ldots,v_i\}]$ for each $1\leq i \leq k $.
Then the orientation $D_i$ of $G_i$ obtained by orienting edges in $E(G_i)\setminus E(G_{i-1})$ arbitrarily as long as $A(D_{i-1}) \subset A(D_{i})$ and $N^+_{D_{i-1}}(u_{i}) \subset N^+ _{D_{i}}(v_{i})$ is competitive for each $1\leq i \leq k $ by Proposition~\ref{prop:adding-vertex}.
Therefore $D_k$ is a competitive orientation of $G'$.
\end{proof}

The following are immediate consequences of Theorem~\ref{thm:making-multi-complete}.
Especially, Corollary~\ref{lem:making-multi-complete} plays a key role throughout this paper.
\begin{Cor} \label{lem:making-multi-complete}
Let $k$ and $l$ be positive integers with $l \geq k \geq 3$; $n_1,\ldots, n_k$ be positive integers such that $n_1 \geq \cdots \geq n_k$;  $n'_1, \ldots ,n'_l$ be positive integers such that
$n'_1 \geq \cdots \geq n'_l$, $n'_1 \geq n_1$,
$n'_2 \geq n_2,\ldots$, and $n'_k \geq n_k$.
If $K_{n_1,\ldots,n_k}$ is competitively orientable,
then $K_{n'_1,\ldots,n'_l}$ is also competitively orientable.
\end{Cor}

\begin{Cor} \label{lem:dividing-partite-set}
Let $k$ be a positive integer with $k \geq 3$; $n_1, \ldots, n_k$,$n'_k,n'_{k+1}$ be positive integers such that $n_k=n'_k+n'_{k+1}$. If $K_{n_1,\ldots,n_{k-1},n_k}$ is competitively orientable, then $K_{n_1,\ldots,n_{k-1},n'_k,n'_{k+1}}$ is also competitively orientable.
\end{Cor}

\subsection{Competitive multipartite tournaments}

\begin{Prop} \label{prop:two-partite-out-neighbor-2-2}
Suppose that $D$ is a competitive multipartite tournament.
If the out-neighbors of a vertex $v$ are included in exactly two partite sets $U$ and $V$ of $D$,
then $|N^+(v) \cap U | \geq 2$ and $|N^+(v) \cap V | \geq 2$.
\end{Prop}

\begin{proof}
Suppose that there exists a vertex $v$ whose out-neighbors are included in exactly two partite sets $U$ and $V$ of $D$.
If
$N^+(v) \cap U = \{u\}$ for some vertex $u$ in $D$, then $u$ is a common out-neighbor of each vertex in $N^+(v) \cap V$ and $v$, and so $u$ has no out-neighbor in $N^+(v)$, which contradicts Proposition~\ref{prop:char-competi-digraph}(2).
Therefore $|N^+(v) \cap U| \geq 2$.
By symmetry, $|N^+(v) \cap V| \geq 2$.
Thus the statement is true.
\end{proof}

\begin{Lem} \label{lem:orientation-8-vertices}
For $k \in \{5,6\}$, if a competitive $k$-partite tournament $D$ of order $8$ has at least two vertices of outdegree at least $4$,
then $k=6$ and $D$ is an orientation of $K_{2,2,1,1,1,1}$
in which there exist exactly two vertices of outdegree at least $4$.
\end{Lem}

\begin{proof}
Suppose that a competitive $k$-partite $D$ has $8$ vertices at least two of which have outdegree at least $4$ for some $k \in \{5,6\}$.
Suppose $k=5$.
It is easy to check that the numbers of arcs in $K_{4,1,1,1,1},K_{3,2,1,1,1}$, and $K_{2,2,2,1,1}$ are $22,24$, and $25$, respectively.
Therefore $|A(D)|$ becomes maximum when $D$ is an orientation of $K_{2,2,2,1,1}$, so $|A(D)| \leq 25$.
By Proposition~\ref{prop:char-competi-digraph}(4), there exist at least $\max \{4|V(D)|-|A(D)|,0\}$ vertices of outdegree $3$ in $D$, so at least $7$ vertices have outdegree $3$ in $D$.
Therefore there exists at most one vertex of outdegree at least $4$, which is a contradiction to the hypothesis.
Thus $k=6$ and $D$ is an orientation of $K_{3,1,1,1,1,1}$ or $K_{2,2,1,1,1,1}$.
If $D$ is an orientation of $K_{3,1,1,1,1,1}$, then $|A(D)|=25$ and so, by the same reason, we reach a contradiction.
Thus $D$ is an orientation of $K_{2,2,1,1,1,1}$.
By the way, $D$ has $8$ vertices and $26$ arcs, so $4|V(D)|-|A(D)| =6$.
	Then there exist at least $6$ vertices of outdegree $3$ by Proposition~\ref{prop:char-competi-digraph}(4).
Therefore $D$ has exactly two vertices of outdegree at least $4$.
\end{proof}

\begin{Thm} \label{thm:complete-condition-outdegree-3}
Suppose that $D$ is a competitive $k$-partite tournament for some integer $k \in \{4,5,6\}$ which has a vertex $u$ of outdegree $3$.
Then $D$ contains a subdigraph isomorphic to the digraph $\tilde{D}$ in Figure~\ref{fig:subdigraph-degree3} and $|V(D)| \geq 9$. In particular, if $k=4$, then $|V(D)| \geq 10 $.
\end{Thm}

\begin{proof}
Each pair of vertices has a common out-neighbor in $D$ since  $D$ is competitive.
Let $N^+(u)=\{v_1,v_2,v_3\}$.
Without loss of generality, we may assume that $C:=v_1 \to v_2 \to v_3 \to v_1 $ is a directed cycle of $D$ by Proposition~\ref{prop:char-competi-digraph}(3).
Let $w_i$ be a common out-neighbor of $v_i$ and $v_{i+1}$ for each $1 \leq i \leq 3$ (identify $v_4$ with $v_1$).
If $w_j=w_k$ for some distinct $j,k \in \{1,2,3\}$,
then $N^+(u)=\{v_1,v_2,v_3\} \subseteq N^-(w_j)$ and so $w_j$ does not share a common out-neighbor with $u$, which is a contradiction.
Therefore $w_1$, $w_2$, and $w_3$ are all distinct.
Moreover, since $u$ and $w_i$ share a common out-neighbor for each $1\leq i \leq 3$,  $\{(w_1,v_3),(w_2,v_1),(w_3,v_2)\}\subset A(D)$.
Thus, so far, we have a subdigraph $\tilde{D}$ of $D$ with the vertex set $\{u,v_1,v_2,v_3,w_1,w_2,w_3\}$ given in Figure~\ref{fig:subdigraph-degree3}.

If $|V(D)|=7$, then
the underlying graph of $D$ must have $21$ edges by Theorem~\ref{thm:char-orient-graph}(3) and so $D$ is a $7$-partite tournament, which is a contradiction.
Thus $|V(D)| \geq 8$.
\begin{figure}
\begin{center}
\begin{tikzpicture}[x=1.0cm, y=1.0cm]
   \tikzset{->-/.style={decoration={
  markings,
  mark=at position #1 with {\arrow{>}}},postaction={decorate}}}

   \vertex (u) at (0,0) [label=above:$u$]{};
   \vertex (v2) at (2,0) [label=left:$v_2$]{};
   \vertex (v3) at (2,2) [label=above:$v_3$]{};
   \vertex (v1) at (2,-2) [label=below:$v_1$]{};
   \vertex (w2) at (4,0) [label=right:$w_2$]{};
    \vertex (w3) at (4,2) [label=above:$w_3$]{};
    \vertex (w1) at (4,-2) [label=below:$w_1$]{};
   \path
(u) edge [->-=0.97,thick] (v1)
(u) edge [->-=0.97,thick] (v3)
(u) edge [->-=0.97, bend left=30, thick] (v2)
(v1) edge [->-=0.97,bend left=15,thick] (v2)
(v3) edge [->-=0.97,thick] (w3)
(v2) edge [->-=0.97,thick] (w1)
(v2) edge [->-=0.97,thick] (w2)
(v3) edge [->-=0.97,thick] (w2)
(v1) edge [->-=0.97,thick] (w3)
(v1) edge [->-=0.97,thick] (w1)
(w3) edge [->-=0.97,thick] (v2)
(w2) edge [->-=0.97,thick] (v1)
(w1) edge [->-=0.97,thick] (v3)
   (v2) edge [->-=0.97,bend left=15,thick] (v3)
   (v3) edge [->-=0.97,bend left=20,thick] (v1)
   ;
   ;
\draw (3, -3) node{$\tilde{D}$};
\end{tikzpicture}
 \end{center}
  \caption{The subdigraph $\tilde{D}$ obtained in the proof of Theorem~\ref{thm:complete-condition-outdegree-3}}
  \label{fig:subdigraph-degree3}
\end{figure}
\begin{figure}
\begin{center}
\begin{tikzpicture}[x=1.0cm, y=1.0cm]
 \tikzset{->-/.style={decoration={
  markings,
  mark=at position #1 with {\arrow{>}}},postaction={decorate}}}
   \vertex (u) at (0,0) [label=above:$u$]{};
   \vertex (v2) at (2,0) [label=left:$v_2$]{};
   \vertex (v3) at (2,2) [label=above:$v_3$]{};
   \vertex (v1) at (2,-2) [label=below:$v_1$]{};
   \vertex (w2) at (4,0) [label=right:$w_2$]{};
    \vertex (w3) at (4,2) [label=above:$w_3$]{};
    \vertex (w1) at (4,-2) [label=below:$w_1$]{};
   \path
(u) edge [->-=0.97,thick] (v1)
(u) edge [->-=0.97,thick] (v3)
(u) edge [->-=0.97, bend left=30, thick] (v2)
(v1) edge [->-=0.97,bend left=15,thick] (v2)
(v3) edge [->-=0.97,thick] (w3)
(v2) edge [->-=0.97,thick] (w1)
(v2) edge [->-=0.97,thick] (w2)
(v3) edge [->-=0.97,thick] (w2)
(v1) edge [->-=0.97,thick] (w3)
(v1) edge [->-=0.97,thick] (w1)
(w3) edge [->-=0.97,thick] (v2)
(w2) edge [->-=0.97,thick] (v1)
(w1) edge [->-=0.97,thick] (v3)
   (v2) edge [->-=0.97,bend left=15,thick] (v3)
   (v3) edge [->-=0.97,bend left=20,thick] (v1)
(w1) edge [->-=0.97,bend right=40,thick] (w3)
    (w2) edge [->-=0.97,thick] (w1)
    (w3) edge [->-=0.97,thick] (w2)
	;
	;
\draw (3, -3) node{$D_1$};
\end{tikzpicture}
\hspace{2em}
\begin{tikzpicture}[x=1.0cm, y=1.0cm]
 \tikzset{->-/.style={decoration={
  markings,
  mark=at position #1 with {\arrow{>}}},postaction={decorate}}}
   \vertex (u) at (0,0) [label=above:$u$]{};
   \vertex (v2) at (2,0) [label=left:$v_2$]{};
   \vertex (v3) at (2,2) [label=above:$v_3$]{};
   \vertex (v1) at (2,-2) [label=below:$v_1$]{};
   \vertex (w2) at (4,0) [label=right:$w_2$]{};
    \vertex (w3) at (4,2) [label=above:$w_3$]{};
    \vertex (w1) at (4,-2) [label=below:$w_1$]{};
     \vertex (x) at (0.6,-1.5) [label=above:$x$]{};
   \path
(u) edge [->-=0.97,thick] (v1)
(u) edge [->-=0.97,thick] (v3)
(u) edge [->-=0.97, bend left=30, thick] (v2)
(v1) edge [->-=0.97,bend left=15,thick] (v2)
(v3) edge [->-=0.97,thick] (w3)
(v2) edge [->-=0.97,thick] (w1)
(v2) edge [->-=0.97,thick] (w2)
(v3) edge [->-=0.97,thick] (w2)
(v1) edge [->-=0.97,thick] (w3)
(v1) edge [->-=0.97,thick] (w1)
(w3) edge [->-=0.97,thick] (v2)
(w2) edge [->-=0.97,thick] (v1)
(w1) edge [->-=0.97,thick] (v3)
   (v2) edge [->-=0.97,bend left=15,thick] (v3)
   (v3) edge [->-=0.97,bend left=20,thick] (v1)
(w1) edge [->-=0.97,bend right=40,thick] (w3)
(x) edge [->-=0.97,thick] (v1)
	;
	;
\draw (3, -3) node{$D_2$};
\end{tikzpicture}
 \end{center}
  \caption{The subdigraphs $D_1$ and $D_2$ considered in the proof of Theorem~\ref{thm:complete-condition-outdegree-3}}
  \label{fig:subdigraph-degree3-example}
\end{figure}
To reach a contradiction, suppose that $|V(D)|=8$.
Then  $V(D)=V(\tilde{D}) \cup \{x\}$ for some vertex $x$ in $D$ and
 \begin{equation}
\label{eq:prop:complete-condition-outdegree-3-0}
|N^+(v_i)|=3 \mbox{ or } 4 \end{equation} for each $1\leq i \leq 3$.
Since $x$ and $u$ must compete and $N^+(u)=\{v_1,v_2,v_3\}$, one of $v_1$, $v_2$, $v_3$ is a common out-neighbor of $u$ and $x$.
Without loss of generality, we may assume $v_1$ is a common out-neighbor of $x$
and $u$.
Then \[N^+(v_1)=\{v_2,w_1,w_3\} \text{ and } (x,v_1)\in A(D).\]
By Proposition~\ref{prop:char-competi-digraph}(3), the out-neighbors of $v_1$ form a directed cycle.
Therefore
$\{v_1,v_2,v_3,w_1,w_3\}$ forms a $5$-tournament in $D$, so \[k \geq 5.\]
By the way, since $(w_3,v_2)$ and $(v_2,w_1)$ are arcs of $D$, \[(w_1,w_3)\in A(D)\]
(see the digraph $D_2$ given in Figure~\ref{fig:subdigraph-degree3-example} for an illustration).
Since $v_1$ and $w_2$ compete and $N^+(v_1)=\{v_2,w_1,w_3\}$,
\begin{equation}
\label{eq:prop:complete-condition-outdegree-3-2}N^+(w_2) \cap \{w_1,w_3\} \neq \emptyset. \end{equation}

 We first claim that $\{v_1,v_2,v_3,w_1,w_2,w_3\}$ forms a tournament in $D$.
  Since $D_2$ is a subgraph of $D$, we need to show that $\{w_1,w_2,w_3\}$ forms a tournament in $D$.
 		As we have shown that $\{v_1,v_2,v_3,w_1,w_3\}$ is a tournament in $D$, it remains to show that $w_2$ is adjacent to $w_1$ and $w_3$ in $D$.
 	
 	Suppose, to the contrary, that there is no arc between $w_1$ and $w_2$.
Then $(w_2,w_3) \in A(D)$ by~\eqref{eq:prop:complete-condition-outdegree-3-2}.
Then the vertices $v_1$, $w_2$, $w_3$ cannot form a directed cycle. Yet,  $v_1$, $w_2$, $w_3$ are out-neighbors of $v_3$, so $|N^+(v_3)|=4$ by \eqref{eq:prop:complete-condition-outdegree-3-0} and Proposition~\ref{prop:char-competi-digraph}(3).
Since $x$ is the only possible new out-neighbor of $v_3$ in $D$,
$N^+(v_3)=\{v_1,w_2,w_3,x\}$.
Since $x$ is the only possible common out-neighbor of $w_2$ and $v_2$, $N^+(w_2) \cap N^+(v_2)=\{x\}$.
Thus $N^+(v_2)=\{x,v_3,w_1,w_2\}$ and $\{v_1,w_3,x\} \subseteq N^+(w_2)$.
Since $v_2$ and $v_3$ have outdegree $4$, $D$ is an orientation of $K_{2,2,1,1,1,1}$ by Lemma~\ref{lem:orientation-8-vertices}.
Then $\{w_1,w_2\}$ forms a partite set of $D$.
Since $u$ has outdegree $3$ in $D$, $(w_2,u) \in A(D)$ and so
$\{u,v_1,w_3,x \} \subseteq N^+(w_2)$.
Then $v_2$, $v_3$, and $w_2$ have outdegree at least $4$, which contradicts Lemma~\ref{lem:orientation-8-vertices}.
Thus there is an arc between $w_1$ and $w_2$.

 Now we suppose, to the contrary, that there is no arc between $w_2$ and $w_3$.
	Then $v_1$, $w_2$, $w_3$ cannot form a directed cycle. Since they are out-neighbors of $v_3$, $|N^+(v_3)|=4$ by \eqref{eq:prop:complete-condition-outdegree-3-0} and Proposition~\ref{prop:char-competi-digraph}(3) and so $N^+(v_3)=\{v_1,w_2,w_3,x\}$.
Since there is no arc between $w_2$ and $w_3$, there is an arc $(w_2,w_1)$ in $D$ by \eqref{eq:prop:complete-condition-outdegree-3-2}.
For the same reason, $x$ is the only possible common out-neighbor of $w_3$ and $v_2$, so $N^+(w_3) \cap N^+(v_2)=\{x\}$.
Thus $v_2$ has outdegree $4$ by~\eqref{eq:prop:complete-condition-outdegree-3-0}.
	Since $v_3$ also has outdegree  $4$,
	$D$ is an orientation of $K_{2,2,1,1,1,1}$ by Lemma~\ref{lem:orientation-8-vertices}.
Thus $\{w_2,w_3\}$ is a partite set of $D$.
Since $N^+(v_1)=\{v_2,w_1,w_3\}$ and $ N^+(w_3) \cap N^+(v_2)=\{x\}$,
$(x,w_1)$ must be an arc of $D$ in order for
$v_1$ and $x$ to compete.
Since $u$ is the only possible common out-neighbor of $x$ and $w_1$, there exist arcs $(x,u)$ and $(w_1,u)$ in $D$.
Then $\{x,u,v_1,v_2,v_3,w_1\}$ forms a tournament and we reach a contradiction to the fact that $D$ is an orientation of $K_{2,2,1,1,1,1}$ with $\{w_2,w_3\}$ as a partite set of $D$.
Therefore $\{v_1,v_2,v_3,w_1,w_2,w_3\}$ forms a tournament as we claimed.
Thus $k =6$ and each of $u$ and $x$ belongs to a partite set of size at least $2$. Furthermore, since $N^+(u)=\{v_1,v_2,v_3\}$, $u$ cannot form a partite set with $v_1$, $v_2$, or $v_3$ and so $u$ and exactly one of $w_1$, $w_2$, and $w_3$ belong to the same partite set.

Suppose, to the contrary, that $(w_3,w_2)\in A(D)$.
Then $(w_2,w_1) \in A(D)$ by~\eqref{eq:prop:complete-condition-outdegree-3-2}.
Therefore $w_1\to w_3 \to w_2 \to w_1$ forms a directed cycle.
Then, for each pair of $w_1$, $w_2$, and $w_3$, $x$ and $u$ are its only possible common out-neighbors.
Since $u$ and one of $w_1$, $w_2$, and $w_3$ belong to the same partite set, exactly one pair of $w_1$, $w_2$, and $w_3$ can prey on $u$.
Then the other two pair of  $w_1$, $w_2$, and $w_3$ prey on $x$.
Therefore $\{w_1,w_2,w_3 \} \subseteq N^-(x)$.
Thus $x$ and exactly one of $v_2$ and $v_3$ belong to the same partite set (recall that we assumed $(x,v_1) \in A(D)$).
Let $v_{j+1}$ be the vertex $j\in \{1,2\}$ belonging to the same partite set with $x$.
Then, since $\{u,v_{j-1}\} \subseteq N^-(v_j)$ (identify $v_0$ with $v_3$) and $\{w_1,w_2,w_3 \} \subseteq N^-(x)$,
$x$ and $v_j$ have no common out-neighbor in $D$, which is a contradiction.
Therefore $(w_3,w_2)\notin A(D)$ and so
 \[(w_2,w_3)\in A(D).\]
 Thus $N^+(w_3) \subseteq \{u,v_2,x\}$ and so, by Proposition~\ref{prop:char-competi-digraph}(3), $N^+(w_3)=\{u,v_2,x\}$.
Then $x$ is the only possible common out-neighbor of each pair of $v_3$ and $w_3$, and $v_2$ and $w_3$.
 Therefore $N^+(v_3)=\{v_1,w_2,w_3,x\}$ and $N^+(v_2)=\{v_3,w_1,w_2,x\}$
 by~\eqref{eq:prop:complete-condition-outdegree-3-0}.
Thus $D$ is an orientation of $K_{2,2,1,1,1,1}$ by Lemma~\ref{lem:orientation-8-vertices}.
Moreover, since $N^+(v_1)=\{v_2,w_3,w_1\}$,
$w_1$ is the only possible common out-neighbor of $x$ and $v_1$ and so
$(x,w_1) \in A(D)$.
Then $u$ must be a common out-neighbor of $w_1$ and $w_3$.
Therefore $\{w_2,u\}$ is a partite sets of size $2$ in $D$. Then, since $\{(w_3,x),(v_2,x),(v_2,x),(x,v_1),(x,w_1)\} \subset A(D)$, $\{x\}$ should be a partite set of $D$ and so $k \geq 7$, which is a contradiction.
Therefore we have shown that $|V(D)| \neq 8$ and so $|V(D)| \geq 9$.

To show the ``particular" part, suppose $k=4$.
Let $V_1,V_2,V_3,$ and $V_4$ be the partite sets of $D$. By Proposition~\ref{prop:char-competi-digraph}(3),
$u$, $v_1$, $v_2$, and $v_3$ belong to distinct partite sets.
Without loss of generality, we may assume that $u \in V_1$, $v_1 \in V_2$, $ v_2 \in V_3$, and $v_3 \in V_4$.
Let $y_i$ be a common out-neighbor of $v_i$ and $w_i$ in $D$ for each $1\leq i \leq 3$.
If $y_1=y_2=y_3$, then $\{v_1,v_2,v_3\} \subseteq N^-(y_1)$, which implies that $u$ and $y_1$ do not share a common out-neighbor, and we reach a contradiction.
Therefore at least two of $y_1$, $y_2$, and $y_3$ are distinct.
In addition,
we may see from a subdigraph $\tilde{D}$ given in Figure~\ref{fig:subdigraph-degree3},
that $\{u,w_1,w_2,w_3\} \subseteq V_1$.
Therefore $y_i$ cannot be $w_j$ for each $1\leq i,j \leq 3$.
Suppose, to the contrary, that $|V(D)|=9$.
Then exactly two of $y_1$, $y_2$, and $y_3$ are the same.
Without loss of generality, we may assume $y_1=y_2$ and $y_1 \neq y_3$.
Neither $v_1$ nor $v_2$ is a common out-neighbor of $y_1$ and $u$.
Thus $v_3$ must be a common out-neighbor of $y_1$ and $u$.
Yet, $y_1$ is a common out-neighbor of $v_1$, $v_2$, $w_1$, and $w_2$, so $y_1 \in V_4$ and we reach a contradiction.
Thus $|V(D)| \geq 10$.
\end{proof}
\section{Proofs}
\subsection{A proof of Theorem~\ref{thm:complete-6-partite}}\label{sec:6-partite}
In this subsection, we characterize complete  $k$-partite graphs which are competitively orientable for an integer $k \geq 6$.

Since $K_{1,1,1,1,1,1,1} \cong K_7$ and a competitive  orientation of $K_7$ is given in Figure~\ref{fig:7-tournament-complete}, the following proposition is immediately true by Corollary~\ref{lem:making-multi-complete}.
\begin{Prop} \label{prop:k>=7-complete}
For each integer $k \ge 7$, any $k$-partite complete graph is competitively orientable.
\end{Prop}

We now have completely characterized sizes of the partite sets of a competitive $k$-partite tournament for $k=2$ or $k \geq 7$
by Corollary~\ref{cor:no-biparte-complete} and Proposition~\ref{prop:k>=7-complete}.
Accordingly, it remains to study competitive $k$-partite tournaments for $ 3 \leq k \leq 6$.
Especially, in the rest of this section, we characterize competitively orientable complete $6$-partite graphs.

\begin{Prop} \label{prop:indegree-condition-complete}
Let $D$ be a competitive $k$-partite tournament  for some positive integer $k\geq 3$ with the partite sets $V_1, \ldots, V_k$.
Then
there exists a competitive $k$-partite tournament $D^*$ with the partite sets $V_1, \ldots, V_k$
such that each vertex in $D^*$ has indegree at least $2$.
\end{Prop}

\begin{proof}
Suppose that there exists a vertex $v_1$ of indegree at most $1$ in $V_{t_1}$ for some $t_1 \in \{1,\ldots,k\}$.
Let $D_1=D-v_1$.
Then $D_1$ is competitive by Proposition~\ref{prop:char-competi-digraph}(1).
By Corollary~\ref{cor:no-biparte-complete},
$D_1$ is not a bipartite tournament.
Suppose that there exists a vertex $v_{2}$ of indegree at most $1$ in $V_{t_2}$ for some $t_2 \in \{1,\ldots,k\}$ in $D_1$.
Let $D_2=D_1-v_{2}$.
Therefore $D_2$ is competitive by Proposition~\ref{prop:char-competi-digraph}(1) and so, by Corollary~\ref{cor:no-biparte-complete}, $D_2$ is not a bipartite tournament.
We keep repeating this process.
Since $D$ has a finite number of vertices, this process terminates to produce digraphs $D_1, D_2, \ldots, D_l$ each of which is competitive and none of which is a bipartite tournament.
Since $D_l$ is competitive, the number of partite sets in $D_l$ is at least $3$.
The fact that the process ended with $D_l$ implies that each vertex in $D_l$ has indegree at least $2$.
As some of partite sets of $D_l$ are proper subsets of corresponding partite sets of $D$,
we need to add vertices to obtain a desired $k$-partite tournament.
Let $X$ be the partite set of $D_{l-1}$ to which $v_l$ belongs.
Then $X \subseteq V_{t_l}$.
In the following, we construct a multipartite tournament $D^*_{l-1}$ from $D_l$ such that
$V(D_{l-1})=V(D^*_{l-1})$, $D_{l-1}$ and $D^*_{l-1}$ have the identical partite sets, and $D^*_{l-1}$ is competitive.
We consider two cases for $X$.

{\it Case 1}. $X=\{v_l\}$.
We take a vertex $v'$ in $D_l$.
Then $v'$ has indegree at least $2$.
Now we add $v_l$ to $D_l$ so that $\{v_l\}$ is a partite set of $D^*_{l-1}$, $v_l$ takes the out-neighbors and the in-neighbors
of $v'$ as its out-neighbors and in-neighbors, respectively, and the remaining out-neighbors and in-neighbors of $v_l$ are arbitrarily taken.
Then the indegree of $v_l$ in $D^*_{l-1}$ is at least $2$.
Moreover,
\[V(D_{l})\cup \{v_l\}=V(D^*_{l-1}), \quad A(D_l) \subset A(D^*_{l-1}), \quad  \text{and} \quad  N^+_{D_l}(v') \subset N^+_{D^*_{l-1}}(v_{l}). \]

{\it Case 2}. $\{v_l\}  \varsubsetneq X$.
Then there exists a vertex $v'$ distinct from $v_{l}$ in $X$.
Since $D_l=D_{l-1}-v_{l}$, $v'$ is a vertex of $D_l$.
Now we add $v_l$ to the partite set of $D_l$ where $v'$ belongs so that
$\{v_{l},v'\}$ is involved in a partite set of $D^*_{l-1}$, $v_l$ takes the out-neighbors and the in-neighbors
of $v'$ as its out-neighbors and in-neighbors, respectively.
Then the indegree of $v_l$ in $D^*_{l-1}$ is at least $2$ since the indegree of $v'$ is at least $2$ in $D_l$.
Moreover,
\[V(D_{l})\cup \{v_l\}=V(D^*_{l-1}), \quad A(D_l) \subset A(D^*_{l-1}), \quad \text{and} \quad  N^+_{D_l}(v') \subseteq N^+_{D^*_{l-1}}(v_{l}). \]
In both cases, $D^*_{l-1}$ is competitive by Proposition~\ref{prop:adding-vertex}.

Now we add $v_{l-1}$ to $D^*_{l-1}$
and apply an argument similar to the above one to obtain competitive multipartite tournament $D^*_{l-2}$ each vertex in which has indegree at least $2$.
We may repeat this process until we obtain a competitive $k$-partite tournament
$D^*_0$ each vertex of which has indegree at least $2$.
Since we added $v_i$ to the partite set of $D^*_i$ which is included in $V_{t_i}$ for each $1\leq i \leq l$,
it is true that the partite sets of $D^*_0$ are the same as $D$.
Thus $D^*_0$ is a desired $k$-partite tournament.
\end{proof}

\begin{Prop} \label{prop:K_411111-no}
The complete $6$-partite graph $K_{4,1,1,1,1,1}$ is not competitively orientable.
\end{Prop}

\begin{proof}
Suppose, to the contrary, that there exists a competitive orientation of $K_{4,1,1,1,1,1}$.
Then, by Proposition~\ref{prop:indegree-condition-complete},  there exists a competitive orientation $D$ of $K_{4,1,1,1,1,1}$ each vertex of which has indegree at least $2$.
Let $V_1,\ldots,V_6$ be the partite sets of $D$ with $|V_1|=4$.
By Proposition~\ref{prop:char-competi-digraph}(3),
each vertex has outdegree at least $3$ in $D$.
Then, since each vertex has indegree at least $2$ in $D$,
\begin{equation} \label{equation:prop:K_411111-no}
|N^+(v)|=3 \quad \text{and} \quad |N^-(v)|=2
\end{equation}
for each vertex $v$ in $V_1$.
By Proposition~\ref{prop:char-competi-digraph}(4), there exist at least $\max \{4|V(D)|-|A(D)|,0\}$ vertices of outdegree $3$ in $D$.
Since $4|V(D)|-|A(D)|=6$,
there exist at least $6$ vertices of outdegree $3$.
Thus at least two vertices of outdegree $3$ do not belong to $V_1$.
Let $u$ be a vertex of outdegree $3$ which is not in $V_1$.
Without loss of generality, we may assume $u \in V_2$.

Let $N^+(u)=\{v_1,v_2,v_3\}$.
By Proposition~\ref{prop:char-competi-digraph}(3),
$N^+(u)$ forms a directed cycle in $D$ and
we may assume $v_1 \to v_2 \to v_3 \to v_1$.
Since each out-neighbor of $u$ has indegree at least $3$ by Theorem~\ref{thm:complete-condition-outdegree-3},
$N^+(u) \cap V_1 = \emptyset$ by~\eqref{equation:prop:K_411111-no}.
Therefore we may assume that $V_3=\{v_1\}$, $V_4=\{v_2\}$, $V_5=\{v_3\}$, $V_6=\{x\}$, and $v_1$ is a common out-neighbor of $x$ and $u$.
Then $\{u,v_3,x\} \subseteq N^-(v_1)$.
Let $w_1$ be a common out-neighbor of $v_1$ and $v_2$.
Then $w_1 \in V_1$.
Therefore, by~\eqref{equation:prop:K_411111-no}, $N^-(w_1)=\{v_1,v_2\}$ and $N^+(w_1)=\{u,v_3,x\}$.
Thus  $N^+(w_1)\subseteq N^-(v_1)$
and so $w_1$ and $v_1$ have no common out-neighbor, which is a contradiction.
\end{proof}

Let $D$ be a digraph with the vertex set $\{v_1,v_2,\ldots,v_n\}$ and $A=(a_{ij})$ be the {\it adjacency matrix }of $D$ such that \begin{displaymath}
a_{ij} = \left\{ \begin{array}{ll}
1 & \textrm{if there is an arc $(v_i,v_j)$ in $D$,}\\
0 & \textrm{otherwise.}\
\end{array} \right.
\end{displaymath}

Now we are ready to prove Theorem~\ref{thm:complete-6-partite}.

\begin{proof}[Proof of Theorem~\ref{thm:complete-6-partite}]
To show the ``only if" part, suppose that there exists a competitive orientation $D$ of $K_{n_1,n_2,\ldots,n_6}$.
We suppose $n_3=1$.

{\it Case 1}. $n_2=1$.
If $n_1 \leq 4$, then
 there exists a competitive orientation of $K_{4,1,1,1,1,1}$ by Corollary~\ref{lem:making-multi-complete}, which contradicts Proposition~\ref{prop:K_411111-no}.
Therefore $n_1 \geq 5$.

{\it Case 2}. $n_2 \geq 2$.
Then $n_1 \geq 2$.
Suppose, to the contrary, that $n_1 = 2$.
Then $n_2=2$, so $D$ is an orientation of $K_{2,2,1,1,1,1}$.
Therefore $4|V(D)|-|A(D)|=6$.
By Proposition~\ref{prop:char-competi-digraph}(4),
there exists a vertex of outdegree $3$ in $D$.
Therefore $|V(D)| \geq 9$ by Theorem~\ref{thm:complete-condition-outdegree-3}, which is a contradiction.
Thus $n_1 \geq 3$.
Hence the ``only if" part is true.

Now we show the ``if" part.
Let $D_\alpha$, $D_\beta$, and $D_\gamma$ be the digraphs whose adjacency matrix are $A_1$, $A_2$, and $A_3$, respectively, given in Figure~\ref{fig:6-partite-complete}.
It is easy to check that
 the inner product of each pair of rows in each matrix is nonzero, so  $D_\alpha$, $D_\beta$, and $D_\gamma$ are competitive.
By applying Corollary~\ref{lem:making-multi-complete} to $D_\alpha$, $D_\beta$, and $D_\gamma$,
we may obtain competitive orientations $D'_\alpha$, $D'_\beta$, and $D'_\gamma$ of $K_{n_1,n_2,n_3,n_4,n_5,n_6}$ for (a) $n_1\geq5$ and $n_2=1$; (b) $n_1 \geq 3$, $n_2 \geq 2$; (c) $n_3\geq 2$, respectively.
Therefore we have shown that the ``if" part is true.
\end{proof}

\begin{figure}
  \centering
    \begin{equation*}
A_1=
\left(\begin{array}{*{10}c}
 0 & 0 & 0 & 0 & 0 &
  0 & 1 & 0 & 1 & 1  \\
  0 & 0 & 0 & 0 & 0 &
  0 & 1 & 1 & 0 & 1  \\
  0 & 0 & 0 & 0 & 0 &
  1 & 0 & 1 & 1 & 0  \\
  0 & 0 & 0 & 0 & 0 &
  1 & 0 & 1 & 0 & 1 \\
  0 & 0 & 0 & 0 & 0 &
  1 & 1 & 0 & 1 & 0 \\
 1 & 1 & 0 & 0 & 0 &
  0 & 1 & 1 & 0 & 0  \\
 0 & 0 & 1 & 1 & 0 &
  0 & 0 & 1 & 1 & 0  \\
 1 & 0 & 0 & 0 & 1 &
  0 & 0 & 0 & 1 & 1  \\
 0 & 1 & 0 & 1 & 0 &
  1& 0 & 0 & 0 & 1  \\
 0 & 0 & 1 & 0 & 1 &
  1 & 1 & 0 & 0 & 0  \\
   \end{array}
    \right)
    \end{equation*}
    \begin{equation*}
A_2=
\left(\begin{array}{*{9}c}
 0 & 0 & 0 &
 1 & 0 &
 1 & 0 & 1 & 0   \\
 0 & 0 & 0 &
 0 & 1 &
 1 & 1 & 0 & 1   \\
 0 & 0 & 0 &
 0 & 1 &
 1 & 0 & 1 & 1   \\
 0 & 1 & 1 &
 0 & 0 &
 0 & 0 & 1 & 1   \\
 1 & 0 & 0 &
 0 & 0 &
 1 & 0 & 0 & 1   \\
 0 & 0 & 0 &
 1 & 0 &
 0 & 1 & 0 & 1   \\
 1 & 0 & 1 &
 1 & 1 &
 0 & 0 & 0 & 0   \\
 0 & 1 & 0 &
 0 & 1 &
 1 & 1 & 0 & 0   \\
 1 & 0 & 0 &
 0 & 0 &
 0 & 1 & 1 & 0   \\
   \end{array}
    \right)
    \end{equation*}
    \begin{equation*}
    A_3=
\left(\begin{array}{*{9}c}
 0 & 0 & 0 & 0 & 0 & 0
 & 1 & 1 & 1   \\
 0 & 0 & 1 & 0 & 1 & 1
 & 0 & 1 & 0   \\
 1 & 0 & 0 & 0 & 1 & 0
 & 1 & 1 & 1   \\
 1 & 1 & 0 & 0 & 1 & 0
 & 0 & 0 & 1   \\
 1 & 0 & 0 & 0 & 0 & 0
 & 1 & 1 & 1   \\
 1 & 0 & 1 & 1 & 0 & 0
 & 1 & 0 & 0   \\
 0 & 1 & 0 & 1 & 0 & 0
 & 0 & 1 & 0   \\
 0 & 0 & 0 & 1 & 0 & 1
 & 0 & 0 & 1   \\
 0 & 1 & 0 & 0 & 0 & 1
 & 1 & 0 & 0   \\
   \end{array}
    \right)
    \end{equation*}
      \caption{The adjacency matrices $A_1$, $A_2$, and $A_3$ which are orientations of $K_{5,1,1,1,1,1}$, $K_{3,2,1,1,1,1}$, and $K_{2,2,2,1,1,1}$, respectively, in the proof of Theorem~\ref{thm:complete-6-partite}}
      \label{fig:6-partite-complete}
    \end{figure}

\subsection{A proof of Theorem~\ref{thm:complete-3-partite}}\label{sec:3-partite}

In this subsection, we characterize complete tripartite graphs which are competitively orientable.

    \begin{Prop}\label{thm4} Let $D$ be a competitive orientation of $K_{n_1, n_2, n_3}$ for some positive integers $n_1, n_2$, and $n_3$. Then $n_i \ge 4$ for each $1\leq i \leq 3$. \end{Prop}
        \begin{proof}
        Let $V_1, V_2,$ and $V_3$ be the partite sets of $D$ with $|V_i|=n_i$ for each $1\leq i \leq 3$.
        Suppose, to the contrary, that $n_j\leq 3$ for some $j \in \{1,2,3\}$.
        Without loss of generality, we may assume that $n_1 \leq 3$.
        Take $v_1 \in V_1$.
        By Proposition~\ref{prop:char-competi-digraph}(2),
        the out-neighbors of each vertex in $D$ are included in at least two partite sets.
        Then, since $D$ is tripartite tournament, the out-neighbors of each vertex in $D$ are included in exactly two partite sets.
        Thus, by Proposition~\ref{prop:two-partite-out-neighbor-2-2}, there are four vertices $u_1, u_2, w_1,$ and $w_2$ such that $\{u_1, u_2\} \subseteq N^+(v_1) \cap V_2$ and $\{w_1, w_2\} \subseteq N^+(v_1) \cap V_3$.
        By the same proposition, there are two vertices $v_2$ and $v_3$ in $V_1$ such that $\{v_2, v_3\} \subseteq N^+(u_1) \cap V_1$.
        Since $(v_1, u_1) \in A(D)$, $v_2$ and $v_3$ are distinct from $v_1$.
        Since $n_1 \leq 3$, $n_1=3$. Then
         $V_1 = \{v_1, v_2, v_3\}$.
        Therefore, by Proposition~\ref{prop:two-partite-out-neighbor-2-2}, $N^+(v) \cap V_1 =\{v_2,v_3\}$ for each vertex $v$ in $N^+(v_1)$.
        This implies that each out-neighbor of $v_1$ has $v_2$ as its out-neighbor.
        Therefore $v_1$ and $v_2$ cannot compete, which is a contradiction.
        \end{proof}

    \begin{Lem}\label{lem6} If $D$ is a competitive orientation of $K_{4,4,4}$ with the partite sets $V_1, V_2$, and $V_3$, then, for distinct $i,j \in \{1,2,3\}$ and each $u \in V_i$, $|N^+(u) \cap V_j|=|N^-(u) \cap V_j|=2$. \end{Lem}
        \begin{proof}
       Suppose that there exists a competitive orientation $D$ of $K_{4,4,4}$ with partite sets $V_1, V_2$, and $V_3$.
         Take distinct $i$ and $j$ in $\{1,2,3\}$.
        Then there are exactly $16$ arcs between $V_i$ and $V_j$.
        On the other hand, by Proposition~\ref{prop:two-partite-out-neighbor-2-2}, for each $u \in V_i$ and $v \in V_j$,
        \begin{equation*}
        |N^+(u) \cap V_j| \geq2 \quad \text{and} \quad |N^+(v) \cap V_i| \geq2.
        \end{equation*}
        Therefore
        \begin{equation*}
        16=\sum_{u \in V_i} |N^+(u) \cap V_j| + \sum_{v \in V_j} |N^+(v) \cap V_i|\geq 16.
        \end{equation*}
        and so $|N^+(u) \cap V_j|=|N^+(v) \cap V_i|=2$ for each $u \in V_i$ and $v \in V_j$.
        Hence $|N^+(u) \cap V_j|=|N^-(u) \cap V_j|=2$ for each $u \in V_i$.
        \end{proof}

    \begin{Lem}\label{lem7} If $D$ is a competitive orientation of $K_{4,4,4}$ with the partite sets $V_1, V_2$, and $V_3$, then, for some distinct $i$ and $j$ in $\{1,2,3\}$, there is a pair of vertices $x$ and $y$ in $V_i$ such that $N^+(x) \cap V_j=N^+(y) \cap V_j$. \end{Lem}
        \begin{proof}
        Suppose that there exists a competitive orientation $D$ of $K_{4,4,4}$ with the partite sets $V_1, V_2$, and $V_3$.
        Suppose, to the contrary, that, for distinct $i,j \in \{1,2,3\}$,

        \begin{equation}\label{lem7:eq1} N^+(u) \cap V_j \neq N^+(v) \cap V_j
        \end{equation}
        for any pair of vertices $u$ and $v$ in $V_i$.
        Fix $i \in \{1,2,3\}$ and $u, v \in V_i$.
        Let $w$ and $z$ be the remaining vertices in $V_i$.
        Since  $D$ is competitive, $u$ and $v$ have a common out-neighbor in $V_j$ for some $j \in \{1,2,3\} \setminus \{i\}$.
        By Lemma~\ref{lem6} and \eqref{lem7:eq1}, $N^+(u) \cap V_j = \{v_1,v_2\}$ and
        \begin{equation}\label{lem7:eq2}
        N^+(v) \cap V_j = \{v_1,v_3\}
        \end{equation}
        for distinct vertices $v_1,v_2$, and $v_3$ in $V_j$.
        Then, by Lemma~\ref{lem6}, $N^+(v_1) \cap V_i= \{w,z\}$.
        Let $v_4$ be the remaining vertex in $V_j$.
        Then, by Lemma~\ref{lem6} again,
        $N^-(u) \cap V_j =\{v_3,v_4\}$ and
        $N^-(v) \cap V_j=\{v_2,v_4\}$.
        Therefore $N^-(v_4) \cap V_i = \{w,z\}$ by the same lemma and so $N^+(v_4) \cap V_i = \{u,v\}$.
  Thus $v_1$ and $v_4$ cannot have a common out-neighbor in $V_i$.
        Hence they have a common out-neighbor in $V_k$ for $k \in \{1,2,3\} \setminus \{i,j\}$.
        By Lemma~\ref{lem6} and \eqref{lem7:eq1} again, $N^+(v_1) \cap V_k = \{w_1,w_2\}$ and $N^+(v_4) \cap V_k = \{w_1,w_3\}$ for distinct vertices $w_1,w_2$, and $w_3$ in $V_k$.
        Let $w_4$ be the remaining vertex in $V_k$.
        Then, by Lemma~\ref{lem6},
        \[N^-(v_1) \cap V_k = \{w_3,w_4\} \quad \text{and} \quad N^-(v_4) \cap V_k = \{w_2,w_4\}.\]
        Meanwhile we note that $w_2$ and $u$ have a common out-neighbor in $V_j$.
        Since $N^+(u)\cap V_j=\{v_1,v_2\}$ and $N^+(v_1) \cap V_k = \{w_1,w_2\}$, $v_2$ is a common out-neighbor of $w_2$ and $u$.
        Then $N^+(w_2) \cap V_j = \{v_2,v_4\}$ by Lemma~\ref{lem6}.
        Thus $w_2$ and $v$ cannot have a common out-neighbor in $V_j$ by \eqref{lem7:eq2}.
        Since $w_2$ and $v$ belong to $V_k$ and $V_i$, respectively, they cannot compete in $D$ and we reach a contradiction.
        \end{proof}

    \begin{Thm}\label{thm8}
     The complete tripartite graph $K_{4,4,4}$ is not competitively orientable.
  \end{Thm}
        \begin{proof}
          Suppose, to the contrary, that there exists a competitive orientation $D$ of $K_{4,4,4}$. Let $V_1,V_2$, and $V_3$ be the partite sets of $D$.
        Then, by Lemmas~\ref{lem6} and~\ref{lem7}, for some distinct $i$ and $j$ in $\{1,2,3\}$, there is a pair of vertices $u_1$ and $u_2$ in $V_i$ such that $N^+(u_1) \cap V_j = N^+(u_2) \cap V_j = \{v_1,v_2\}$ for some vertices $v_1$ and $v_2$ in $V_j$. Without loss of generality, we may assume that $i=1$ and $j=2$.
        Let $u_3$ and $u_4$ (resp.\ $v_3$ and $v_4$) be the remaining vertices in $V_1$ (resp.\ $V_2$).
        Then, by Lemma~\ref{lem6},
        \[N^+(u_3) \cap V_2 = N^+(u_4) \cap V_2 = \{v_3,v_4\},\]
        \[N^+(v_1) \cap V_1 = N^+(v_2) \cap V_1 = \{u_3,u_4\},\]
        \[N^+(v_3) \cap V_1 = N^+(v_4) \cap V_1 = \{u_1,u_2\}\]
        (see Figure~\ref{thm8:fig1} for an illustration).
\begin{figure}
\begin{center}
\begin{tikzpicture}[auto,thick, scale=0.9]
    \tikzstyle{player}=[minimum size=5pt,inner sep=0pt,outer sep=0pt,draw,circle]
    \tikzstyle{source}=[minimum size=5pt,inner sep=0pt,outer sep=0pt,ball color=black, circle]
    \tikzstyle{arc}=[minimum size=5pt,inner sep=1pt,outer sep=1pt, font=\footnotesize]
    \tikzset{->-/.style={decoration={
  markings,
  mark=at position #1 with {\arrow{>}}},postaction={decorate}}}

    \draw (180:3.2cm) node (name) {$V_1$};
    \draw (0:3.2cm) node (name) {$V_2$};

    \draw (270:3.2cm) node (name) {the arcs from $V_1$ to $V_2$};

    \path (-2,1.5)    node [player,label=left:$u_1$]  (u1) {};
    \path (-2,0.5)   node [player,label=left:$u_2$]  (u2) {};
    \path (-2,-0.5)   node [player,label=left:$u_3$]  (u3) {};
    \path (-2,-1.5)   node [player,label=left:$u_4$]  (u4) {};
    \path (2,1.5)   node [player,label=right:$v_1$]  (v1) {};
    \path (2,0.5)   node [player,label=right:$v_2$]  (v2) {};
    \path (2,-0.5)   node [player,label=right:$v_3$]  (v3) {};
    \path (2,-1.5)   node [player,label=right:$v_4$]  (v4) {};

    \draw[black,thick,->-=0.85] (u1) -- (v1);
    \draw[black,thick,->-=0.85] (u1) - + (v2);
    \draw[black,thick,->-=0.85] (u2) - + (v1);
    \draw[black,thick,->-=0.85] (u2) - + (v2);
    \draw[black,thick,->-=0.85] (u3) - + (v3);
    \draw[black,thick,->-=0.85] (u3) - + (v4);
    \draw[black,thick,->-=0.85] (u4) - + (v3);
    \draw[black,thick,->-=0.85] (u4) - + (v4);

\end{tikzpicture}
\hspace{2em}
\begin{tikzpicture}[auto,thick, scale=0.9]
    \tikzstyle{player}=[minimum size=5pt,inner sep=0pt,outer sep=0pt,draw,circle]
    \tikzstyle{source}=[minimum size=5pt,inner sep=0pt,outer sep=0pt,ball color=black, circle]
    \tikzstyle{arc}=[minimum size=5pt,inner sep=1pt,outer sep=1pt, font=\footnotesize]
    \tikzset{->-/.style={decoration={
  markings,
  mark=at position #1 with {\arrow{>}}},postaction={decorate}}}

    \draw (180:3.2cm) node (name) {$V_1$};
    \draw (0:3.2cm) node (name) {$V_2$};

    \draw (270:3.2cm) node (name) {the arcs from $V_2$ to $V_1$};

    \path (-2,1.5)    node [player,label=left:$u_1$]  (u1) {};
    \path (-2,0.5)   node [player,label=left:$u_2$]  (u2) {};
    \path (-2,-0.5)   node [player,label=left:$u_3$]  (u3) {};
    \path (-2,-1.5)   node [player,label=left:$u_4$]  (u4) {};
    \path (2,1.5)   node [player,label=right:$v_1$]  (v1) {};
    \path (2,0.5)   node [player,label=right:$v_2$]  (v2) {};
    \path (2,-0.5)   node [player,label=right:$v_3$]  (v3) {};
    \path (2,-1.5)   node [player,label=right:$v_4$]  (v4) {};

    \draw[black,thick,->-=0.85] (v1) -- (u3);
    \draw[black,thick,->-=0.85] (v1) - + (u4);
    \draw[black,thick,->-=0.85] (v2) - + (u3);
    \draw[black,thick,->-=0.85] (v2) - + (u4);
    \draw[black,thick,->-=0.85] (v3) - + (u1);
    \draw[black,thick,->-=0.85] (v3) - + (u2);
    \draw[black,thick,->-=0.85] (v4) - + (u1);
    \draw[black,thick,->-=0.85] (v4) - + (u2);

\end{tikzpicture}
\caption{The arcs between $V_1$ and $V_2$}\label{thm8:fig1}
\end{center}
\end{figure}
        Therefore each of the following pairs does not have a common out-neighbor in $V_2$: $\{u_1,u_3\}$; $\{u_1,u_4\}$; $\{u_2,u_3\}$; $\{u_2,u_4\}$.
        In addition, each of the following pairs does not have a common out-neighbor in $V_1$: $\{v_1,v_3\}$; $\{v_1,v_4\}$; $\{v_2,v_3\}$; $\{v_2,v_4\}$. Then each of these pairs has a common out-neighbor in $V_3$.
        Let $w_1,w_2,w_3$, and $w_4$ be the common out-neighbors of $\{u_1,u_3\},\{u_1,u_4\},\{u_2,u_3\}$, and $\{u_2,u_4\}$, respectively.
        Then, by Lemma~\ref{lem6}, $w_i \neq w_j$ for distinct $i,j \in \{1,2,3,4\}$ and so $V_3=\{w_1,w_2,w_3,w_4\}$.
        Without loss of generality, we may assume that $w_1,w_2,w_3$, and $w_4$ are the common out-neighbors of $\{v_1,v_3\},\{v_1,v_4\},\{v_2,v_3\}$, and $\{v_2,v_4\}$, respectively.
        Then, by Lemma~\ref{lem6}, $w_1$ and $w_2$ are out-neighbors of $u_1$ and $w_3$ and $w_4$ are out-neighbors of $v_2$ in $V_3$, so $u_1$ and $v_2$ do not compete in $D$, which is a contradiction.
        \end{proof}

Now we are ready to prove Theorem~\ref{thm:complete-3-partite}.

\begin{proof}[Proof of Theorem~\ref{thm:complete-3-partite}]
To show the ``only if" part, suppose that
$D$ is a competitive orientation of $K_{n_1,n_2,n_3}$.
Then, by Proposition~\ref{thm4},
$n_i \geq 4$ for each $1\leq i \leq 3$.
If $n_1=4$, then $n_2=n_3=4$, which contradicts Theorem~\ref{thm8}.
Therefore $n_1 \geq 5$ and so the ``only if" part is true.

Now we show the ``if" part.
Let $D_\alpha$ be the digraph whose adjacency matrix is $A_4$ given in Figure~\ref{fig:3-partite-complete}.
It is easy to check that $D_\alpha$ is an orientation of $K_{5,4,4}$ and
the inner product of each pair of rows in each matrix is nonzero, so $D_\alpha$ is competitive.
If $n_1\geq5$ and $n_3\geq 4$, then, by applying Corollary~\ref{lem:making-multi-complete} to $D_\alpha$,
we obtain a competitive orientation $D'_\alpha$ of $K_{n_1,n_2,n_3}$.
\end{proof}
        \begin{figure}
    \centering
    \begin{equation*}
A_4=
\left(\begin{array}{*{13}c}
 0&0&0&0&0&0&1&0&1&0&1&1&0\\
 0&0&0&0&0&0&1&1&0&1&0&0&1\\
 0&0&0&0&0&1&0&0&1&1&0&0&1\\
 0&0&0&0&0&1&0&1&0&0&1&0&1\\
 0&0&0&0&0&1&0&1&0&1&0&1&0\\
 1&1&0&0&0&0&0&0&0&1&1&0&0\\
 0&0&1&1&1&0&0&0&0&1&1&0&0\\
 1&0&1&0&0&0&0&0&0&0&0&1&1\\
 0&1&0&1&1&0&0&0&0&0&0&1&1\\
 1&0&0&1&0&0&0&1&1&0&0&0&0\\
 0&1&1&0&1&0&0&1&1&0&0&0&0\\
 0&1&1&1&0&1&1&0&0&0&0&0&0\\
 1&0&0&0&1&1&1&0&0&0&0&0&0
\\
     \end{array}
    \right)
    \end{equation*}
  \caption{
  The adjacency matrix $A_4$ which is an orientation of $K_{5,4,4}$ in the proof of Theorem~\ref{thm:complete-3-partite}
  }\label{fig:3-partite-complete}
\end{figure}

\subsection{Proofs of  Theorems~\ref{thm:complete-4-partite} and \ref{thm:complete-5-partite}}\label{sec:4,5-partite}
In this subsection, we characterize complete $k$-partite graphs which are competitively orientable for the cases $k=4$ and $k=5$.

\begin{Prop} \label{prop:n.b.of.size.1}
Any competitive $k$-partite tournament for $k \in \{4,5\}$ has at most $k-3$ singleton partite sets.
\end{Prop}

\begin{proof}
Let $V_1, V_2, \ldots, V_k$ be the partite sets of a competitive $k$-partite tournament $D$ for some $k \in \{4.5\}$.
We may assume that $x \in V_1$, $y \in V_2$, and $z \in V_3$.

We suppose $k=4$.
To reach a contradiction, suppose that there are at least $2$ partite sets of size $1$.
Without loss of generality, we may assume $|V_1|=|V_2|=1$.
Then $V_1=\{x\}$ and $V_2=\{y\}$.
Without loss of generality, we may assume $z$ is a common out-neighbor of $x$ and $y$.
Then $N^+(z) \subseteq V_4$, which contradicts Proposition~\ref{prop:char-competi-digraph}(2).
Therefore $D$ has at most $1$ partite set of size $1$.

Suppose $k=5$.
To reach a contradiction, suppose that there are at least $3$ partite sets of size $1$.
Without loss of generality, we may assume $|V_1|=|V_2|=|V_3|=1$.
Then $V_1=\{x\}$, $V_2=\{y\}$, and $V_3=\{z\}$.
Suppose that $x$ and $y$ have a common out-neighbor $w$ in $V_4\cup V_5$.
Without loss of generality, we may assume $w\in V_4$.
Then $N^+(w) \subseteq V_3 \cup V_5$.
By Proposition~\ref{prop:char-competi-digraph}(2),
$N^+(w) \cap V_3 \neq \emptyset $ and $N^+(w) \cap V_5 \neq \emptyset $.
However, $N^+(w) \cap V_3=\{z\}$, which contradicts Proposition~\ref{prop:two-partite-out-neighbor-2-2}.
Thus $w \notin V_4\cup V_5$ and so $w=z$.
By symmetry, the only possible common out-neighbor of $y$ and $z$ is $x$.
Since $z\in N^+(x)$, $x$ cannot be an out-neighbor of $z$ and we reach a contradiction.
Therefore $D$ has at most $2$ partite sets having size $1$.
\end{proof}

By Theorem~\ref{thm:complete-condition-outdegree-3},
the out-neighbors of a vertex of outdegree $3$ in a competitive $k$-partite tournament for some $k \in \{4,5\}$ form a directed cycle and we have the following lemma.
\begin{Lem} \label{lem:degree3-partite-set-sizes}
Let $D$ be a competitive $k$-partite tournament for some $4 \leq k \leq 5$.
Suppose that a vertex $u$ has outdegree $3$.
If $N^+(u)\subseteq U \cup V \cup W$ for distinct partite sets $U$, $V$, and $W$ of $D$,
then $|U|+ |V|+ |W| \leq |V(D)|-4$.
\end{Lem}

\begin{proof}
Suppose that $N^+(u)\subseteq U \cup V \cup W$ for distinct partite sets $U$, $V$, and $W$ of $D$.
Since $u$ has outdegree $3$,
by Theorem~\ref{thm:complete-condition-outdegree-3},
$D$ contains a subdigraph isomorphic to $\tilde{D}$ given in Figure~\ref{fig:subdigraph-degree3}.
We may assume that the subdigraph is $D_1$ itself including labels.
We may assume $v_1 \in U$, $v_2 \in V$, and $v_3 \in W$.
Then $\{u,w_1,w_2,w_3\} \cap (U\cup V \cup W) = \emptyset$.
Thus $|V(D) \setminus (U\cup V \cup W)| \geq 4$ and so $|U|+ |V|+ |W|=|U\cup V \cup W| \leq |V(D)| -4$.
\end{proof}

\begin{Cor} \label{cor:no-or-3,3,2,2}
Neither $K_{3,3,2,2}$ nor $K_{3,3,3,1}$ is  competitively orientable.
\end{Cor}
\begin{proof}
Suppose, to the contrary, that there exists a competitive orientation $D$ of $K_{3,3,2,2}$ or $K_{3,3,3,1}$. Then $|A(D)| < 40$.
If each vertex in $D$ has outdegree at least $4$, then $|A(D)| \geq 40$, which is a contradiction.
Therefore there exists a vertex $u$ of outdegree $3$, then, the out-neighbors of $u$ belong to three distinct partite sets $U$, $V$, and $W$ by Proposition~\ref{prop:char-competi-digraph}(3) and, by Lemma~\ref{lem:degree3-partite-set-sizes},
$|U|+|V|+|W| \leq |V(D)| - 4=6$, which is impossible.
\end{proof}

\begin{Lem} \label{lem:complete-3-partite}
Let $n_1$, $n_2$, and $n_3$ be positive integers such that $n_1 \geq n_2 \geq n_3$.
If $K_{n_1,n_2,n_3,1}$ is competitively orientable,
then $n_3 \geq 3$.
\end{Lem}
\begin{proof}
Suppose that there exists a competitive orientation $D$ of $K_{n_1,n_2,n_3,1}$.
Then $n_3 \geq 2$ by Proposition~\ref{prop:n.b.of.size.1}.
Suppose, to the contrary, that $n_3=2$.
Let $V_1, \ldots, V_4$ be the partite sets of $D$ satisfying $|V_1|=n_1$, $|V_2|=n_2$, $|V_3|=2$, and $|V_4|=1$.
Let $V_3=\{x_1,x_2\}$ and $y$ be a common out-neighbor of $x_1$ and $x_2$.
Then $V_3 \cap N^+(y) = \emptyset$, so, by Proposition~\ref{prop:char-competi-digraph}(2), $N^+(y)$ is included in exactly two partite sets.
If $y \in V_1\cup V_2$, then
$|N^+(y) \cap V_4 |=1$, which contradicts Proposition~\ref{prop:two-partite-out-neighbor-2-2}.
Therefore $y \notin V_1\cup V_2$ and so
$y \in V_4$. Thus $V_4=\{y\}$.
Hence $N^+(y) \subseteq V_1 \cup V_2$.
Take a vertex $u$ in $N^+(y)$.
Then $N^+(u) \subseteq V_1 \cup V_3$ or $V_2 \cup V_3$.
Therefore $N^+(u)$ is included in exactly two partite sets by Proposition~\ref{prop:char-competi-digraph}(2).
Since $|V_3|=2$,
 $N^+(u)\cap V_3=V_3$, that is, $u$ is a out-neighbor of neither $x_1$ nor $x_2$, by Proposition~\ref{prop:two-partite-out-neighbor-2-2}.
Since $u$ was arbitrarily chosen in $N^+(y)$,
any out-neighbor of $y$ is a out-neighbor of neither $x_1$ nor $x_2$.
Thus $x_1$ and $y$ have no common out-neighbor in $D$, which is a contradiction.
Hence $n_3 \neq 2$ and so $n_3 \geq 3$.
\end{proof}
Now we are ready to show Theorem~\ref{thm:complete-4-partite}.
\begin{proof}[Proof of Theorem~\ref{thm:complete-4-partite}]
To show the ``only if" part,
suppose that $D$ is a competitive orientation of $K_{n_1,n_2,n_3,n_4}$.

{\it Case 1}. $n_4=1$.
Then $n_3 \geq 3$ by Lemma~\ref{lem:complete-3-partite}, so $n_1 \geq 3$.
If $n_1 =3$, then $n_1=n_2=n_3=3$ and so $D$ is an orientation of $K_{3,3,3,1}$, which contradicts Corollary~\ref{cor:no-or-3,3,2,2}.
Therefore $n_1 \geq 4$.

{\it Case 2}. $n_4 \geq 2$.
Then $n_3 \geq 2$ and so (c) holds.
Suppose $n_3=2$. Then $n_4=2$.
If $n_1=3$, then, by applying Corollary~\ref{lem:making-multi-complete} to $D$, we obtain a competitive orientation $D^*$ of $K_{3,3,2,2}$, which contradicts Corollary~\ref{cor:no-or-3,3,2,2}.
Therefore $n_1 \geq 4$.
Thus the ``only if" part is true.

Now we show the ``if" part.
Let $D_\alpha$, $D_\beta$, $D_\gamma$ be the digraphs whose adjacency matrices are $A_5$, $A_6$, and $A_7$, respectively, given in Figure~\ref{fig:4-partite-complete}.
It is easy to check that $D_\alpha$, $D_\beta$, and $D_\gamma$ are orientations of $K_{4,3,3,1}$, $K_{4,2,2,2}$, and $K_{3,3,3,2}$, respectively, and the inner product of each pair of rows in each matrix is nonzero, so  $D_\alpha$, $D_\beta$, and $D_\gamma$ are competitive.
By applying Corollary~\ref{lem:making-multi-complete} to $D_\alpha$, $D_\beta$, and $D_\gamma$,
we may obtain orientations $D'_\alpha$, $D'_\beta$, and $D'_\gamma$ of $K_{n_1,n_2,n_3,n_4}$ each of which are competitive for (a) $n_1 \geq 4$, $n_3 \geq 3$, and $n_4 \geq 1$; (b) $n_1 \geq 4$, $n_3 =2$, and $n_4 =2$;
(c) $n_3 \geq 3$ and $n_4 \geq 2$, respectively.
Therefore we have shown that the ``if" part is true.
\end{proof}

   \begin{figure}
  \centering
    \begin{equation*}
A_5=
\left(\begin{array}{*{11}c}
  0& 0& 0& 0& 1& 0& 0& 1& 0& 0& 1\\
  0& 0& 0& 0& 0& 1& 0& 0& 1& 0& 1\\
  0& 0& 0& 0& 1& 0& 1& 0& 1& 1& 0\\
  0& 0& 0& 0& 0& 1& 1& 1& 0& 1& 0\\
  0& 1& 0& 1& 0& 0& 0& 1& 1& 0& 0\\
  1& 0& 1& 0& 0& 0& 0& 1& 1& 0& 0\\
  1& 1& 0& 0& 0& 0& 0& 0& 0& 1& 1\\
  0& 1& 1& 0& 0& 0& 1& 0& 0& 0& 1\\
  1& 0& 0& 1& 0& 0& 1& 0& 0& 0& 1\\
  1& 1& 0& 0& 1& 1& 0& 0& 0& 0& 0\\
  0& 0& 1& 1& 1& 1& 0& 0& 0& 1& 0\\
     \end{array}
        \right)
    \end{equation*}
    \begin{equation*}
    A_6=
\left(\begin{array}{*{10}c}
 0 & 0 & 0 & 0 &
 1 & 0
 & 1 & 0
 & 1 & 0   \\
 0 & 0 & 0 & 0 &
 0 & 1
 & 0 & 1
 & 1 & 0   \\
  0 & 0 & 0 & 0 &
 1 & 0
 & 0 & 1
 & 0 & 1   \\
  0 & 0 & 0 & 0 &
 0 & 1
 & 1 & 0
 & 0 & 1   \\
   0 & 1 & 0 & 1 &
 0 & 0
 & 1 & 1
 & 0 & 0   \\
   1 & 0 & 1 & 0 &
 0 & 0
 & 1 & 1
 & 0 & 0   \\
 0 & 1 & 1 & 0 &
 0 & 0
 & 0 & 0
 & 1 & 1   \\
 1 & 0 & 0 & 1 &
 0 & 0
 & 0 & 0
 & 1 & 1   \\
 0 & 0 & 1 & 1 &
 1 & 1
 & 0 & 0
 & 0 & 0   \\
 1& 1 & 0 & 0 &
 1 & 1
 & 0 & 0
 & 0 & 0   \\
   \end{array}
    \right)
    \end{equation*}
     \begin{equation*}
 A_7=
\left(\begin{array}{*{11}c}
0&0&0&1&0&0&0&0&1&1&1\\
0&0&0&0&1&0&0&1&0&1&1\\
0&0&0&1&0&1&1&0&1&0&1\\
0&1&0&0&0&0&0&1&1&1&0\\
1&0&1&0&0&0&0&1&0&0&1\\
1&1&0&0&0&0&1&0&0&1&0\\
1&1&0&1&1&0&0&0&0&0&0\\
1&0&1&0&0&1&0&0&0&1&0\\
0&1&0&0&1&1&0&0&0&0&1\\
0&0&1&0&1&0&1&0&1&0&0\\
0&0&0&1&0&1&1&1&0&0&0\\
  \end{array}
    \right)
    \end{equation*}
      \caption{The adjacency matrices $A_5$, $A_6$, $A_7$ which are orientations of $K_{4,3,3,1}$, $K_{4,2,2,2}$, and $K_{3,3,3,2}$ in the proof of Theorem~\ref{thm:complete-4-partite}.}
      \label{fig:4-partite-complete}
    \end{figure}

In the following, we study $5$-partite tournaments which are competitive.

\begin{Thm} \label{thm:no-orientation-5-partite-32211}
The complete $5$-partite graph $K_{3,2,2,1,1}$
is not competitively orientable.
\end{Thm}

\begin{proof}
Suppose, to the contrary, that there exists a competitive orientation $D$ of $K_{3,2,2,1,1}$.
Let $V_1,\ldots,V_5$ be the partite sets of $D$ with $|V_1|=3$, $|V_2|=|V_3|=2$, and $|V_4|=|V_5|=1$.
Since $4|V(D)|-|A(D)|=5$,
\begin{itemize}
\item[($\dagger$)]
there exist at least $5$ vertices of outdegree $3$ in $D$
\end{itemize}
 by Proposition~\ref{prop:char-competi-digraph}(4).
Take a vertex $u$ of outdegree $3$.
Then $D$ contains a subdigraph containing $u$ isomorphic to $\tilde{D}$ given in Figure~\ref{fig:subdigraph-degree3} by Theorem~\ref{thm:complete-condition-outdegree-3}.
We may assume that the subdigraph is $D_1$ itself including labels.
For each $i=1,2,3$,
since $w_i$ is adjacent to each of $v_1$, $v_2$ and $v_3$ in $D$,
\begin{itemize}
\item[($\S$)] $w_i$ cannot belong to a partite set containing an out-neighbor of $u$.
\end{itemize}
By Proposition~\ref{prop:char-competi-digraph}(3),
the out-neighbors
$v_1,v_2$, and $v_3$ of $u$ belong to three distinct partite sets $U$, $V$, and $W$.
By Lemma~\ref{lem:degree3-partite-set-sizes},
$|U|+|V|+|W| \leq |V(D)| - 4=5$.
Therefore
\begin{equation}
\label{eq:thm:no-orientation-5-partite-32211-0-0}
|N^+(u) \cap V_1|=|N^+(u) \cap V_4|= |N^+(u) \cap V_5|=1
 \end{equation}
 or
 \begin{equation}
 \label{eq:thm:no-orientation-5-partite-32211-0-1}
  |N^+(u) \cap V_i|=|N^+(u) \cap V_j|=|N^+(u) \cap
  V_k|=1
  \end{equation}
 for $2 \leq i < j < k \leq 5$.
We first show that each vertex in $V_4 \cup V_5$ has outdegree at least $4$.

Suppose, to the contrary, that
 $V_4 \cup V_5$ contains a vertex of outdegree at most $3$. Then, by Proposition~\ref{prop:char-competi-digraph}(3), the vertex has outdegree $3$.
We may regard it as $u$ since $u$ is a vertex of outdegree $3$ arbitrarily chosen.
Without loss of generality, we may assume $u \in V_5$.
Then $N^+(u) \cap V_5 =\emptyset$.
Therefore \eqref{eq:thm:no-orientation-5-partite-32211-0-0} cannot happen and so
\eqref{eq:thm:no-orientation-5-partite-32211-0-1} holds.
Thus, without loss of generality, we may assume that $v_1 \in V_2$, $v_2 \in V_3$, and $v_3 \in V_4$.
By ($\S$),
$V_1=\{w_1,w_2,w_3\}$.
Let $V_2=\{v_1,x_1\}$ and $V_3=\{v_2,x_2\}$.
Then \[N^-(u)=\{w_1,w_2,w_3,x_1,x_2\}.\]
Since $x_1$ is the only possible common out-neighbor of each of pairs $\{v_2,w_2\}$ and $\{v_2,w_3\}$,
\begin{equation} \label{eq:prop:no-orientation-5-partite-32211-1}
\{v_2,w_2,w_3\} \subseteq N^-(x_1).\end{equation}
In addition, $x_2$ is the only possible common out-neighbor of each of pairs $\{v_1,w_1\}$ and $\{v_1,w_2\}$.
Therefore $\{v_1,w_1,w_2\} \subseteq N^-(x_2).$
By the way, $x_1$ and $x_2$ are the only possible common out-neighbors of $v_3$ and $w_3$.
If $x_2$ is a common out-neighbor of $v_3$ and $w_3$,
then $\{v_1,v_3,w_1,w_2,w_3\} \subseteq N^-(x_2)$ and so $N^+(x_2)\subseteq \{u,x_1\}$, which contradicts Proposition~\ref{prop:char-competi-digraph}(3).
Therefore $x_1$ is a common out-neighbor of $v_3$ and $w_3$.
Then $\{v_2,v_3,w_2,w_3\} \subseteq N^-(x_1)$ by~\eqref{eq:prop:no-orientation-5-partite-32211-1}, so
 $N^+(x_1) \subseteq \{u,w_1,x_2\}$.
 Thus $N^+(x_1) = \{u,w_1,x_2\}$ by Proposition~\ref{prop:char-competi-digraph}(3).
However, since $\{w_1 ,x_2\} \subset N^-(u)$,
$N^+(x_1)$ cannot form a directed cycle, which contradicts Proposition~\ref{prop:char-competi-digraph}(3).
Hence $u \notin V_4 \cup V_5$ and we reach a contradiction.
Therefore
\begin{equation} \label{eq:prop:no-orientation-5-partite-32211-2}
|N^+(v)| \geq 4
\end{equation}
for each vertex $v$ in $V_4 \cup V_5$.

Now we show that each of $V_2$ and $V_3$ has exactly one vertex of outdegree $3$, which implies that each vertex of $V_1$ has outdegree $3$.
Since $D$ has at least $5$ vertices of outdegree $3$ by ($\dagger$), $V_2 \cup V_3$ has at least $2$ vertices of outdegree $3$ by~\eqref{eq:prop:no-orientation-5-partite-32211-2}.
Take a vertex of outdegree $3$ in $ V_2 \cup V_3$. Then we may regard it as $u$.
Take a vertex $v$ of outdegree $3$ distinct from $u$ in $V_2 \cup V_3$.
Then, $N^+(x) \cap V_2 =\emptyset$ or $N^+(x) \cap V_3 =\emptyset$ for each vertex $x$ in $\{u,v\}$, so, by~\eqref{eq:thm:no-orientation-5-partite-32211-0-0} and \eqref{eq:thm:no-orientation-5-partite-32211-0-1},
\[ |N^+(u) \cap V_i|=|N^+(u) \cap V_4|=|N^+(u) \cap V_5|=1 \]
for some $i \in \{1,2,3\}$
and
\[ |N^+(v) \cap V_{j}|=|N^+(v) \cap V_{4}|=|N^+(v) \cap V_{5}|=1 \]
for some $j \in \{1,2,3\}$.
Thus, since $|V_4|=|V_5|=1$,
the vertices in $V_4 \cup V_5$ are common out-neighbors of $u$ and $v$ and we may assume that
\[ N^+(v)=\{v'_1,v_2,v_3\}\]
for some vertex $v'_1$ in $D$, $V_4=\{v_2\}$, and $V_5=\{v_3\}$ by symmetry.
Since $(v_2,v_3)\in A(D)$,
\[ (v'_1,v_2) \in A(D)\]
by Proposition~\ref{prop:char-competi-digraph}(3).
Therefore
\begin{equation} \label{eq:prop:no-orientation-5-partite-32211-001}
\{v,v'_1\} \subseteq N^-(v_2).
\end{equation}
To reach a contradiction, we suppose that $u$ and $v$ are contained in the same partite set.
Without loss of generality, we may assume $\{u,v\}\subseteq V_2$, Then $V_2=\{u,v\}$.
Suppose $v_1 =v'_1$. Then $N^+(u)=N^+(v)$ and $N^-(u)=N^-(v)$.
Therefore any pair of vertices having $v$ as a common out-neighbor has $u$ as a common out-neighbor.
Then, since  $D$ is competitive, $D-u$ is competitive.
However, $D-u$ is an orientation of $K_{3,1,2,1,1}$, which contradicts Proposition~\ref{prop:n.b.of.size.1}.
Therefore $v_1 \neq v'_1$.
Thus \[N^+(u) \cap N^+(v)=\{v_2,v_3\}.\]
If $v_1 \in V_1$, then
$N^+(u)\subset V_1 \cup V_4 \cup V_5$ and so, by ($\S$), $v=w_i$ for some $i \in \{1,2,3\}$, which contradicts $\{v_2,v_3\}\subseteq N^+(v)$.
Therefore $v_1 \notin V_1$ and so $v_1 \in V_3$.
Thus $N^+(u)\subset V_3\cup V_4 \cup V_5$ and so $V_1=\{w_1,w_2,w_3\}$ by ($\S$).
We may show that, by applying the same argument to $v'_1$, $v'_1 \notin V_1$.
Then $v'_1 \in V_3$, so
$\{v_1,v'_1 \} \subseteq  V_3$.
Therefore $V_3=\{v_1,v'_1 \} $.
We know from $D_1$ that $\{u,v_1,w_3\} \subseteq N^-(v_2) $.
Moreover, $\{v,v'_1\} \subseteq N^-(v_2)$ by~\eqref{eq:prop:no-orientation-5-partite-32211-001}.
Thus
 $\{u,v,v_1,v'_1,w_3 \} \subseteq N^-(v_2)$ and so $N^+(v_2) \subseteq \{v_3, w_1,w_2\}$.
Hence $N^+(v_2) =\{v_3, w_1,w_2\}$ by Proposition~\ref{prop:char-competi-digraph}(3).
However, $w_1$ and $w_2$ belong to the same partite set $V_1$, which contradicts Proposition~\ref{prop:char-competi-digraph}(3).
Therefore $u$ and $v$ belong to the distinct partite sets.
Since $u$ and $v$ were vertices of outdegree $3$ arbitrarily chosen, each of $V_2$ and $V_3$ has at most one vertex of outdegree $3$.
By the way, $V_2 \cup V_3$ has at least $2$ vertices of outdegree $3$, so
we may conclude that
each of $V_2$ and $V_3$ has exactly one vertex of outdegree $3$.
Thus each vertex of $V_1$ has outdegree $3$ by ($\dagger$).

Without loss of generality,
we may assume that $u \in V_2$, $v \in V_3$, and
\[(u,v) \in A(D). \]
Then $v_1=v$.
Therefore
$ |N^+(u) \cap V_3|=|N^+(u) \cap V_4|=|N^+(u) \cap
  V_5|=1 $ by~\eqref{eq:thm:no-orientation-5-partite-32211-0-1}.
If $w_2$, which is a common out-neighbor of $v_2$ and $v_3$, is contained in $V_1$,
then $N^+(w_2) \subseteq V_2 \cup V_3$ and so, by Proposition~\ref{prop:two-partite-out-neighbor-2-2}, $w_2$ has outdegree at least $4$, which is a contradiction to the fact that each vertex of $V_1$ has outdegree $3$.
Therefore $w_2 \in V_2$ by ($\S$).
Then
\[V_2=\{u,w_2\}.\]
Thus $\{w_1,w_3 \} \subset V_1$ by ($\S$) and so each of $w_1$ and $w_3$ has outdegree $3$.
Let
\[V_1=\{w_1,w_3,z\} \quad \text{and} \quad V_3=\{v_1,y\}\] for some vertices $y$ and $z$ in $D$.
 We know from $\tilde{D}$ given in Figure~\ref{fig:subdigraph-degree3} that $N^+(w_1) \cap \{v_1,v_2,v_3\}=\{v_3\}$ and $N^+(w_3)\cap \{v_1,v_2,v_3\}=\{v_2\}$.
Since each of $w_1$ and $w_3$ has outdegree $3$, the out-neighbors of $w_i$ belong to distinct partite sets for $i=1,3$ by Proposition~\ref{prop:char-competi-digraph}(3). By recalling that $N^+(u)=\{v_1,v_2,v_3\}$, we may conclude that  $N^+(w_1)=\{u,v_3,y\}$ and $N^+(w_3)=\{u,v_2,y\}$.
Since $(u,v_2) \in A(D)$ and $(u,v_3)\in A(D)$, $(v_2,y)\in A(D)$ and $(v_3,y)\in A(D)$ by the same lemma.
Therefore $\{v_2,v_3,w_1,w_3\} \subseteq N^-(y)$ and so $N^+(y) \subseteq \{u,w_2,z\}$.
Thus $N^+(y)= \{u,w_2,z\}$ by Proposition~\ref{prop:char-competi-digraph}(3).
However, there is no arc between $u$ and $w_2$ and so $N^+(y)$ cannot form a directed cycle, which contradicts Proposition~\ref{prop:char-competi-digraph}(3).
\end{proof}

\begin{figure}
  \centering
    \begin{equation*}
     A_8=
\left(\begin{array}{*{10}c}
  0& 0& 0& 1& 0& 0& 0& 1& 1& 1\\
  0& 0& 0& 0& 1& 1& 0& 1& 0& 1\\
  0& 0& 0& 1& 0& 1& 1& 0& 1& 0\\
  0& 1& 0& 0& 0& 0& 0& 0& 1& 1\\
  1& 0& 1& 0& 0& 0& 0& 1& 1& 0\\
  1& 0& 0& 0& 0& 0& 1& 0& 0& 1\\
  1& 1& 0& 1& 1& 0& 0& 0& 0& 0\\
  0& 0& 1& 1& 0& 1& 0& 0& 0& 1\\
  0& 1& 0& 0& 0& 1& 1& 1& 0& 0\\
  0& 0& 1& 0& 1& 0& 1& 0& 1& 0
   \end{array}
    \right)
    \end{equation*}
  \begin{equation*}
    A_9=
\left(\begin{array}{*{9}c}
 0 & 0 & 1 &0 &1 &0 &0 &0 &1 \\
  0 & 0 & 0 &1 &0 &1 &0 &0 &1 \\
  0&1&0&0&0&0&1&0&1\\
  1&0&0&0&0&0 &0 &1 &1\\
  0&1 &1 &1 &0 &0 &0 &1 &0\\
  1 &0 &1 &1 &0 &0 &1 &0 &0\\
  1 &1 &0 &1 &1 &0 &0 &0 &0\\
  1 &1 &1 &0 &0 &1 &0 &0 &0\\
  0 &0 &0 &0 &1 &1 &1 &1 &0\\
   \end{array}
    \right)
    \end{equation*}
      \caption{The adjacency matrices $A_8$ and $A_9$ which are orientations of $K_{3,3,2,1,1}$, $K_{2,2,2,2,1}$ respectively, in the proof of Theorem~\ref{thm:complete-5-partite}.}
      \label{fig:5-partite-complete}
    \end{figure}

Now we are ready to prove Theorem~\ref{thm:complete-5-partite}.

\begin{proof}[Proof of Theorem~\ref{thm:complete-5-partite}]
To show the ``only if'' part, suppose that there exists a competitive orientation $D$ of $K_{n_1,n_2,\ldots,n_5}$.
By Proposition~\ref{prop:n.b.of.size.1},
 \[n_3 \geq 2.\]

If $n_4 \geq 2$, then (c) holds.
Now suppose $n_4=1$.
 Then $n_5=1$.
 Suppose, to the contrary, that $n_1=2$.
 Then $D$ is an orientation of $K_{2,2,2,1,1}$.
Since $4|V(D)|-|A(D)|=7 > 0$, $D$ has a vertex of outdegree 3 by Proposition~\ref{prop:char-competi-digraph}(4).
Therefore $|V(D)| \geq 9$ by Theorem~\ref{thm:complete-condition-outdegree-3}, which is impossible.
Thus \[n_1 \geq 3.\]
If $n_1 \geq 4$, then (b) holds.
Now suppose $n_1=3$.
Then $n_2 \leq 3$.
If $n_2=2$, then $n_3=2$ and so $D$ is an orientation of $K_{3,2,2,1,1}$, which contradicts Theorem~\ref{thm:no-orientation-5-partite-32211}.
Therefore
$n_2 = 1$ or $3$.
Then, since $n_2 \geq n_3 \geq 2$,
$n_2 =3$ and (a) holds.
Thus we have shown the ``only if" part.

Now we show the ``if" part.
Let $D_\alpha$, $D_\gamma$ be the digraphs whose adjacency matrix are $A_8$ and $A_9$, respectively, given in Figure~\ref{fig:5-partite-complete}.
Then $D_\alpha$ and $D_\gamma$ are orientations of $K_{3,3,2,1,1}$ and $K_{2,2,2,2,1}$, respectively.
Let $D^*_\beta$ be the digraph whose adjacency matrix is $A_6$ given in Figure~\ref{fig:4-partite-complete}.
It is easy to check that the inner product of each pair of rows in each matrix is nonzero, so  $D_\alpha$, $D_\gamma$, and $D^*_\beta$ are competitive.
By Lemma~\ref{lem:dividing-partite-set}, we obtain a competitive orientation $D_\beta$ of $K_{4,2,2,1,1}$ from $D^*_\beta$.
By applying Corollary~\ref{lem:making-multi-complete} to $D_\alpha$, $D_\beta$, and $D_\gamma$,
we may obtain orientations $D'_\alpha$, $D'_\beta$ , and $D'_\gamma$ of $K_{n_1,n_2,\ldots,n_5}$ each of which are competitive for (a) $n_1=3$, $n_2=3$, $n_3 \geq 2$, $n_4=1$, and $n_5=1$;
(b) $n_1\geq 4$, $n_2 \geq n_3 \geq 2$, $n_4 =1$, and $n_5=1$;
(c) $n_4 \geq 2$, respectively.
Therefore we have shown that the ``if" part is true.
\end{proof}
\section{Closing remarks}
\label{sec:completecompetition}
By Corollary~\ref{cor:no-biparte-complete},
there is no complete graph that is the competition graph of a bipartite tournament.
For an integer $k \geq 3$,
Proposition~\ref{prop:k>=7-complete}, and Theorems~\ref{thm:complete-3-partite},~\ref{thm:complete-4-partite},~\ref{thm:complete-5-partite},~\ref{thm:complete-6-partite} may be summarized in the aspect of the number of vertices of a complete graph which is the competition graph of a $k$-partite tournament as follows.

\begin{Thm} \label{Thm:summary:K_n-k-partite}
A complete graph $K_n$ is the competition graph of a $k$-partite tournament for some integer $k\geq 3$ if and only if
\[
 \begin{cases}
    n \geq 13  & \mbox{ if $k=3$;}   \\
  n \geq 10 & \mbox{ if $k=4$;}
\\
 n \geq 9 & \mbox{ if $k\in \{5,6\}$;} \\
 n \geq k  & \mbox{ if $k\geq 7$.}
  \end{cases}
 \]

\end{Thm}

\begin{proof}
For an integer $k\geq 3$, suppose that
a complete graph $K_n$ is the competition graph of a $k$-partite tournament which is an orientation of $K_{n_1,n_2,\ldots,n_k}$.
Then it is competitive and $n_1 + n_2 + \cdots + n_k=n$.
If $k=3$, then $\sum_{i=1}^{3}n_i \geq 13$ by Theorem~\ref{thm:complete-3-partite}.
If $k=4$, then $\sum_{i=1}^{4}n_i \geq 10$ by Theorem~\ref{thm:complete-4-partite}.
If $k=5$, then $\sum_{i=1}^{5}n_i \geq 9$ by Theorem~\ref{thm:complete-5-partite}.
If $k=6$, then $\sum_{i=1}^{6}n_i \geq 9$ by Theorem~\ref{thm:complete-6-partite}.
If $k\geq 7$, then $\sum_{i=1}^{k}n_i \geq k$ by Proposition~\ref{prop:k>=7-complete}.
Each of the above theorems also guarantees the existence of a competitive $k$-partite tournament for the corresponding $k$, so the ``if" part is true.
\end{proof}

As we mentioned previously, there is no graph of order $n$ which is competitively orientable for any integer $3 \le n \le 6$.
For $n \ge 7$, if a graph $G$ of order $n$ is competitively orientable, then $G$ must have at least $3n$ edges by Theorem~\ref{thm:char-orient-graph}(3).
Furthermore, we showed that for each $m \geq 7$, there is a competitively orientable graph of order $m$ with exactly $3m$ edges  in Remark~\ref{rmk:edges-vertices}.
However, for a complete multipartite graph, we doubt that there is a proper spanning subgraph which is competitively orientable because the matrices which we adapted to construct competitive orientations of complete multipartite graphs seem to represent minimal competitive digraphs.

\section{Acknowledgement}
This research was supported by the National Research Foundation of Korea(NRF)  (NRF-2017R1E1A1A03070489 and 2016R1A5A1008055) funded by the Korea government(MSIP).


\end{document}